\documentclass[12pt]{article}
\usepackage{amsmath}
\usepackage{amssymb}
\usepackage{amsthm}
\usepackage{mathrsfs}
\usepackage{appendix}
\usepackage{hyperref}
\usepackage{geometry}
\usepackage{authblk}
\usepackage{tikz}
\usepackage{titlesec}
\usepackage{stmaryrd}

\usetikzlibrary{positioning, shapes.geometric}

\newtheorem{definition}{Definition}
\newtheorem{lemma}{Lemma}
\newtheorem{theorem}{Theorem}
\newtheorem{proposition}{Proposition}
\newtheorem{remark}{Remark}[section]

\newcommand{\e}{\mathrm{e}}
\renewcommand{\i}{\mathrm{i}}
\renewcommand{\d}{\mathrm{d}}

\title{Nekhoroshev type stability for non-local semilinear Schr\"odinger equations }
\date{}
\author[1]{Bingqi Yu}
\author[,1,2]{Yong Li\thanks{Corresponding author}}
\affil[1]{School of Mathematics, Jilin University, Changchun 130012, People’s Republic of China.}
\affil[2]{Center for Mathematics and Interdisciplinary Sciences,	Northeast Normal University, Changchun 130024, People’s Republic of China.}
\allowdisplaybreaks

\begin{document}
	
	\maketitle
	\footnote{E-mail address: \url{yubq23@mails.jlu.edu.cn}(B. Yu),  \url{liyong@jlu.edu.cn}(Y. Li)}
	
	\begin{abstract}
	 This paper investigates the Nekhoroshev-type stability of solutions with ultra-differentiable regularity in Schr\"odinger equations involving non-local nonlinear terms, using the method of rational normal forms.
	 We establish the first rigorous results on logarithmic ultra-differentiable regularity for infinite-dimensional Hamiltonian systems without external parameters.
	 Under Gevrey class regularity assumptions, we achieve the stability times matching Bourgain's conjectured optimal stability time in \cite{B04}. Furthermore, we introduce a novel global vector field norm adapted to the rational normal form framework. This norm eliminates the need for degree tracking during the iteration process, thereby enabling a unified treatment of nonlinear terms.

	\end{abstract}
	

	\tableofcontents
	\section{Introduction}
	\subsection{Main Results}
	In this paper, we consider the following non-local semilinear Schr\"odinger equation:
	\begin{equation}\label{main}
		\i u_t(x,t)+\Delta u(x,t)+u(x,t)\int_{\mathbb{T}} K(x-y)|u(y,t)|^2\d y=0.
	\end{equation}
	The initial data is given as $u(x,0)=u_0$. The equation is defined on $(x,t)\in\mathbb{T}^{\mathsf{d}}\times\mathbb{R}$.
    For any function $u(x)$ in $L^2$, we define its Fourier coefficients as follows:
    
    $$u=\sum_{k\in\mathbb{Z}}u_k\e^{\i kx},\quad u_k=\int_{\mathbb{T}^{\mathsf{d}}}u(x)\e^{-\i kx} \d x.$$
    We consider two classes of function spaces:  the Gevrey space:
    $$W_{s,\mathfrak{g}}^{G}=\{u\mid\Vert u\Vert_s:=\sum_{j\in\mathbb{Z}}|u_j|^2\e^{2s|j|^{\mathfrak{g}}}<\infty\},$$ and another class of spaces, the logarithmic ultra-differentiable functions, which are defined as follows:
	$$W_{s,\theta}^{U}=\{u\mid\Vert u\Vert_s:=\sum_{j\in\mathbb{Z}}|u_j|^2\e^{2s\ln|j|^{\theta}}<\infty\}.$$

    The parameters satisfy $s > 0$, $0 < \mathfrak{g} < 1$, and $\theta > 0$. Functions in these spaces exhibit regularity that is strictly intermediate between smooth and analytic.
    In the sequel, we denote either of these spaces by $W_s$.
	
	For a sufficiently large positive integer \( M \), we define the truncated space as:
	$$W_s^M = \{u \in W_s \mid u_k = 0 \text{ for } |k| \geq M \}.$$
	The space \( W_s^M \) consists of functions in \( W_s \) whose Fourier coefficients vanish for all \( |k| \geq M \), providing a finite-dimensional approximation of \( W_s \).
	The associated projection is defined as:
	$$\Pi^M : W_s \to W_s^M, \quad \Pi^M((u_j)_{j \in \mathbb{Z}}) = (u_j)_{|j| < M}.$$
	This projection extracts the Fourier modes corresponding to indices \( |j| < M \), effectively reducing the infinite-dimensional space \( W_s \) to its finite-dimensional subspace \( W_s^M \).
	Let \( meas(\cdot) \) denote the Lebesgue measure on the finite-dimensional subspace \( W^M_s \).

	 The kernel \( K \) in the non-local interaction term \( K * |u|^2 \) is characterized by its Fourier coefficients:
	 $$K(x) = \sum_{k \in \mathbb{Z}} K_k \e^{\i kx}, \quad K_k = \int_{\mathbb{T}^{\mathsf{d}}} K(x) \e^{-\i kx} \d x.$$
	 Here, \( K_k \) represents the \( k \)-th Fourier mode of \( K(x) \).
	 
	We consider two physically significant cases for the kernel \( K \):
	\[K_k = \frac{1}{|k|^p} \quad \text{for } k \neq 0, \quad K_0= 0,\]and\[	K_k = \e^{-|k|^{\beta}},\]
	where \( p > 0 \) and \( \beta > 0 \). The first case arises in models where the interaction weakens with distance in Fourier space, while the second describes kernels with exponentially decaying modes.
	
	The main theorems for the four possible combinations of \( K_k \)  and the function spaces \( W_s^G \) and \( W_s^U \) are stated below.

	 \begin{theorem}(Main Theorem 1)\label{thm1}
	 	Let the kernel $K_k$ be given as $K_k=\frac{1}{|k|^p},$ for $p\in\mathbb{Z}^+$ and $|k|\neq0$, with $K_0=0.$ Suppose the initial data satisfies  $u(0)\in W_s= W_{s,\mathfrak{g}}^G,s> s_0,0<\mathfrak{g}<1.$ There exists a set $\mathfrak{R}_{\gamma}\subset W_s^M$ and a threshold $r_0$ such that, for $r<r_0$, the solution $u(t)$ of equation \eqref{main} with initial data $u(0)\in B_s(r)\setminus(\Pi^M)^{-1}\mathfrak{R}_{\gamma}$, satisfies
	 	$$\Vert u(t)\Vert_s\leq2\Vert u(0)\Vert_s$$ for
	 	$$|t|\leq\exp(C_{gp}\frac{|\ln r|^2}{\ln|\ln r|}).$$
	 	
	 	Furthermore, the Lebesgue measure satisfies the estimate:
	 	$$\text{meas}(B_s^M(r) \cap \mathfrak{R}_{\gamma}) \leq r^a \, \text{meas}(B_s^M(r)),$$
	 	for any constant \( a < \frac{2}{5} \).
	 \end{theorem}

	 	This result achieves the stability time predicted by Bourgain in \cite{B04}, which is of the order $\exp(c \frac{(\ln r)^2}{\ln |\ln r|} )$. Stability times of this magnitude were previously established for finite-range coupling systems in \cite{B88}. Furthermore, we obtain a sharper measure estimate $\varepsilon^{\frac{2}{5}}$ than previous results, detail compares can be seen in Section \ref{sec:con}.
	 	
	 \begin{theorem}(Main Theorem 2)\label{thm2}
	 	Let the kernel $K_k$ be given as $K_k=\e^{-|k|^{\beta}}$ with $\beta\geq1.$ Suppose the initial data satisfies $u(0)\in W_s= W_{s,\mathfrak{g}}^G,s> s_0,0<\mathfrak{g}<1.$ There exists a set $\mathfrak{R}_{\gamma}\subset W_s^M$ and a threshold $r_0$ such that for $r<r_0$ and the solution $u(t)$ of equation \eqref{main} with initial data $u(0)\in B_s(r)\setminus(\Pi^M)^{-1}\mathfrak{R}_{\gamma}$  satisfies
	 	$$\Vert u(t)\Vert_s\leq2\Vert u(0)\Vert_s$$ for
	 	
	 	$$|t|\leq\exp(C_{gb}\frac{|\ln r|^2}{\ln|\ln r|}).$$
	 	Furthermore, $$\text{meas}(B_s^M(r)\cap\mathfrak{R}_{\gamma})\leq r^a\text{meas}(B_s^M(r)), $$
	 	for any constant $a<\frac{2}{5}$.	
	 \end{theorem}
	 
	 \begin{theorem}(Main Theorem 3)\label{thm3}
	 	Let the kernel \( K_k \) be given as \( K_k = \frac{1}{|k|^p} \) for \( p \in \mathbb{N} \) and \( |k| \neq 0 \), with \( K_0 = 0 \). Suppose the initial data satisfies \( u(0) \in W_s = W_{s,\theta}^U \), where \( s > s_0 \) and \( \theta > 1 \). 
	 	
	 	There exists a set \( \mathfrak{R}_{\gamma} \subset W_s^M \) and a threshold \( r_0 \) such that, for \( r < r_0 \), the solution \( u(t) \) of equation \eqref{main} with initial data \( u(0) \in B_s(r) \setminus (\Pi^M)^{-1} \mathfrak{R}_{\gamma} \), satisfies
	 	$$\Vert u(t) \Vert_s \leq 2 \Vert u(0) \Vert_s \quad \text{for} \quad |t| \leq \exp \left( C_{\theta p} |\ln r|^{\frac{2\theta}{\theta+1}} \right).$$
	 	
	 	Furthermore, the Lebesgue measure satisfies the estimate:
	 	$$\text{meas}(B_s^M(r) \cap \mathfrak{R}_{\gamma}) \leq r^a \, \text{meas}(B_s^M(r)),$$
	 	for any constant \( a < \frac{2}{5} \).
	 \end{theorem}

	 \begin{theorem}(Main Theorem 4)\label{thm4}
	 	Let the kernel \(K_k\) be given as $K_k=\e^{-|k|^{\beta}}$, and $\beta\geq1$. Suppose the initial data satisfies $u(0)\in W_{s,\theta}^U,s> s_0,\theta>1$. There exists a set $\mathfrak{R}_{\gamma}\subset W_s^M$ and a threshold $r_0$ such that for $r<r_0$ and the solution $u(t)$ of equation \eqref{main} with initial data $u(0)\in B_s(r)\setminus(\Pi^M)^{-1}\mathfrak{R}_{\gamma}$  satisfies
	 	$$\Vert u(t)\Vert_s\leq2\Vert u(0)\Vert_s$$ for
	 	$$|t|\leq\exp(C_{\theta b}|\ln r|^{\frac{2\theta}{\theta+1}}).$$
	 	Furthermore, $$\text{meas}(B_s^M(r)\cap\mathfrak{R}_{\gamma})\leq r^a\text{meas}(B_s^M(r)), $$
	 	where $a<\frac{2}{5}.$	
	 \end{theorem}

	\subsection{Nekhoroshev Stability in Infinite Dimensions}
	The classic Nekhoroshev theorem provides exponential stability times for all initial data within a domain in finite-dimensional Hamiltonian systems subject to small analytic perturbations. This result, originally established by Nekhoroshev \cite{N77}, represents a cornerstone in stability theory for Hamiltonian systems. This represents a significant qualitative improvement over the polynomial stability times typically obtained via averaging methods, which generally exhibit slower decay rates for perturbative systems. Crucially, the theorem provides a framework for addressing questions such as the long-term stability of the solar system. It offers effective stability guarantees even for initial conditions lying outside the strictly non-resonant sets assumed in KAM theory. Subsequently, the Nekhoroshev theorem has also been extensively generalized to more general finite-dimensional Hamiltonian systems. We do not attempt to comprehensively enumerate the following results here \cite{B99,B16,B20,YL25}.
	
	In recent years, there has been significant interest in extending Nekhoroshev stability to infinite-dimensional Hamiltonian systems. This shift is motivated by the relevance of such systems to many important physical problems, including partial differential equations with Hamiltonian structures. In addition, there are also many results that extend the KAM theory to infinite-dimensional Hamiltonian systems, namely Hamiltonian PDEs, such as in the literature \cite{KP96,B98,BBM14}. However, extending the classical theory to this setting confronts a significant mathematical challenge: the density of resonances among infinitely many frequencies, which complicates the analysis of stability. Substantial progress has been made in analyzing stability times under various regularity assumptions. Specifically, the dependence of stability times on differentiability, Gevrey regularity, and ultra-differentiability has been extensively studied. For instance, finite differentiability yields stability times of polynomial order in \( \frac{1}{\varepsilon} \), as demonstrated in \cite{BDGS07,BFM24,BFG20,BG24,BMP19,BMP20,FM23}. For systems with Gevrey regularity, stability times improve significantly to the order \( \exp(|\ln \varepsilon|^c) \) for some \( c > 0 \), as demonstrated in \cite{FG13,CLW24}. 
    Recent results also address ultra-differentiable  regularity. For most of these equations, the analysis crucially relies on the introduction of infinitely many external parameters, such as those arising from convolution potentials or inherent non-resonance properties. However, these results establish stability only for a  set of specific non-resonant frequencies of large measure and do not resolve the stability question for the original equation itself.	
	
	\subsection{Internal Parameter Results for Infinite-Dimensional Nekhoroshev Stability}
	Addressing the stability of the original equation without modification requires overcoming the intrinsic problem of resonances in the system. At this point, traditional canonical techniques cannot be applied. In \cite{FM24}, effective stability in the case of complete resonance was obtained using  para-differential techniques and suitable tame estimates. Another effective approach to addressing this issue involves extracting parameters through frequency modulation of the nonlinear integrable component.
	In this framework, the amplitudes of the initial data themselves are utilized as internal parameters, a concept first proposed in \cite{B00}.
    In contrast to external parameters, internal parameters are typically much smaller in magnitude. This necessitates the formulation of more refined non-resonance conditions and requires make estimates for the measure of the resulting non-resonant set, accounting for the comparatively small scale of internal parameters. To address this, we employ the method of rational normal forms. This approach deviates from standard normalization by permitting terms that explicitly involve the solution \( u \) in the denominators of the homological equations, rather than relying exclusively on coefficient matching. This technique, originally introduced in \cite{BFG20b}, has since been applied to a range of problems involving infinite-dimensional Hamiltonian systems, as demonstrated in \cite{BG21,LX24}.
    Although studies concerning the regularity requirements in this context remain limited, results analogous to those in the external parameter case have been established. Polynomial stability times \( (\frac{1}{\varepsilon})^s \) have been proven for Sobolev spaces (finite differentiability) in \cite{BFG20b,BG21}, while sub-exponential stability times \( \exp(|\ln \varepsilon|^\alpha) \) have been obtained under Gevrey regularity in \cite{BCGW24,LX24}.

	\subsection{Non-local Semilinear Schr\"odinger Equation}
	Equation \eqref{main} models non-local interactions characterized by varying degrees of spatial concentration and has broad applications. The physical nature of the interaction is dictated by the decay rate of the Fourier coefficients of the convolution kernel.

	\textbf{Polynomial Decay in Frequency Space}
	This case corresponds to long-range algebraic interactions in the spatial domain, where the interaction strength decays polynomially with distance. A notable example is the Schrödinger-Newton (or Schr\"odinger-Poisson) system. Specifically, for \( p = 2 \) in three dimensions, the convolution kernel corresponds to the Newtonian gravitational potential, which is extensively used to model self-gravitating quantum matter, such as boson stars and semi-classical gravity systems \cite{P95}.  
	
	In the context of two-dimensional quantum materials, such as graphene or 2D electron gases, the Coulomb interaction potential \( V(x) = 1/|x| \) exhibits a similarly long-range behavior. In frequency space, this potential transforms into \( \hat{V}_k \sim \frac{1}{|k|} \), highlighting the algebraic decay of interactions \cite{CGPNG09,KUP12}.

	\textbf{Exponential Decay in Frequency Space}
	
	The exponential decay in frequency space indicates that the spatial kernel functions as a highly smoothing operator, effectively suppressing high-frequency components while enhancing the regularity of the solution. For example, thermal nonlinear optical media are predominantly described by diffusion equations, which produce Yukawa-type response functions. In the regime of strong non-locality, however, the response function can be well-approximated by a Gaussian profile \cite{RCMSC05}. This Gaussian approximation is consistent with the exponential decay assumption adopted in our model. Specifically, for such highly non-local media, the kernel in frequency space typically takes the form \( K_k \sim \e^{-k^2} \), as demonstrated in \cite{KBRW01}.

	\subsection{Article Structure}
	The structure of this paper is outlined as follows. Section \ref{sec:con} demonstrates the main contributions of this article. Section \ref{sec:set} introduces the main notation and functional setting that will be used consistently throughout the paper.
	Section \ref{sec:res} consists of two subsections. Subsection \ref{sec:pol} introduces the class of polynomials defined on the space from Section \ref{sec:set}, and Subsection \ref{sec:nfl} establishes a key lemma regarding their transformation into resonant normal forms. Section \ref{sec:tru} examines the properties of truncated polynomials on finite-dimensional subspaces. This analysis lays the foundation for proving the finite-dimensional fractional normal form lemma in the following section.
	Section \ref{sec:int} is composed of two subsections. Subsection \ref{sec:fre} introduces a norm defined for fractional expressions and establishes its fundamental operational properties, which will facilitate its application in subsequent proofs. Subsection \ref{sec:infl} demonstrates that truncated polynomials can be transformed into an integrable fractional normal form with a remainder term. Section \ref{sec:sta} establishes stability estimates through the bootstrap method, providing crucial bounds necessary for the subsequent analysis. Section \ref{sec:mea} derives measure estimates for the set of initial data that satisfy the required non-resonance conditions. These estimates demonstrate that the stability results remain valid for a set of initial values of large measure. Finally, Section \ref{sec:len} focuses on the explicit calculation of the stability time exponents, highlighting their role in quantifying the timescales of stability in the system.

	\section{Main Contributions}\label{sec:con}
	This paper examines Nekhoroshev-type stability for Schrödinger equations, focusing on systems without external parameters and incorporating various types of nonlinear non-local terms. We derive stability times within Gevrey regularity classes and logarithmic ultra-differentiable classes, thereby extending previous results to a broader class of systems with more general non-local nonlinearities.

    \textbf{Vector Field Norm for Fractional Normal Form}  
    
A key methodological contribution of this work is the introduction of a novel vector field norm, specifically designed to streamline the rational normal form framework and address challenges associated with the treatment of rational nonlinearities.  

	In earlier studies on rational normal forms (e.g., \cite{BFG20,BFG20b,BG21,BCGW24}), the estimates were primarily conducted using coefficient-wise norms. For instance, the commonly employed norm  
	\[ 
	\Vert P\Vert = \sup_{J} |P_J|, \quad P = \sum_{J} P_J u^J,  
	\]  
	relied on monitoring the coefficients \( P_J \), which correspond to individual monomials \( u^J \). This traditional coefficient-based approach necessitated meticulous tracking of the polynomial degrees of both numerators and denominators at each iterative step of solving homological equations, often leading to highly intricate combinatorial estimates.  
	
	By contrast, the vector field norm introduced in this paper directly applies to vector fields, thereby circumventing the need for detailed tracking of polynomial degrees in both numerators and denominators. This not only unifies the treatment of polynomial and rational nonlinearities but also significantly simplifies and accelerates the iterative proof of the normal form lemma. Moreover, it reduces the combinatorial complexity inherent to traditional methods, offering a more efficient and streamlined framework.

    \textbf{Improved Measure Estimates for Initial Data}
    
    For systems with internal parameters, the initial values act as critical parameters for tuning the non-resonance frequencies. As a result, the stability conclusions may not hold universally for all initial values. However, it is possible to obtain results that are valid for the vast majority of initial values in a measure-theoretic sense. Therefore, precise measure estimates on the set of initial values where the conclusions hold are essential for evaluating the stability of systems with internal parameters.
    
    In this paper, we provide sharper measure estimates for favorable nonlinear terms. Specifically, in all four scenarios considered, the measure of the resonant set is bounded by 
    \(    \varepsilon^{\frac{2}{5}}    \).  This result improves upon previous estimates, such as   \(  \varepsilon^{\frac{1}{3}}\) for the NLS equation in \cite{BFG20b},    \(    \varepsilon^{\frac{1}{35}}
    \) for the KdV and BO equations in \cite{BG21}, and 
    \(\varepsilon^{\frac{1}{6}}\) for the Schrödinger-Poisson equation in \cite{BCGW24}.
    
    \textbf{Internal Parameters in Ultra-Differentiable Classes}
    
    A key contribution of our work lies in addressing the stability problem in logarithmic ultra-differentiable regularity---a concept introduced by Z. Hani---within the challenging framework of internal parameters.
    
    While prior studies on systems with internal parameters have primarily concentrated on classical Gevrey or analytic classes, recent progress in \cite{CCMW22} has expanded this analysis to ultra-differentiable stability. These studies, however, relied on external parameters and achieved stability timescales of order 
    \[
    T_r \sim \exp(|\ln r|^{7/6}) \quad \text{for } \theta = 2.
    \]
    
    To our knowledge, no previous work has investigated the impact of this regularity on stability times in cases where frequency modulation depends on initial data. Addressing this gap, we establish new stability results and achieve a lifespan of order 
    \[
    \exp(|\ln r|^{\frac{2\theta}{\theta+1}})
    \]
    for general 
    \(\theta > 1\). 
    
    Our main contribution lies in developing stability results within the context of internal parameters, addressing previously unresolved questions in the literature. Further details are provided in Theorems~\ref{thm3} and~\ref{thm4}.

    \section{Setting}\label{sec:set}
    Equation \eqref{main} can be formalized as an infinite-dimensional Hamiltonian system on the function space \( H^1(\mathbb{T}^{\mathsf{d}}) \), equipped with the standard symplectic form:  
    $$\Omega(u, v) = \text{Im} \int_{\mathbb{T}^{\mathsf{d}}} u(x)\overline{v(x)}\d x,$$  
    where \( \text{Im} \) denotes the imaginary part. The associated Hamiltonian is expressed as:  
    \begin{equation*}
    	H(u) = \int_{\mathbb{T}^{\mathsf{d}}}|\nabla u(x)|^2\d x + \iint_{\mathbb{T}^{\mathsf{d}}\times\mathbb{T}^{\mathsf{d}}}|u(x)|^2 K(x-y) |u(y)|^2 \d y \d x,
    \end{equation*}  
    where \( K(x-y) \) represents the non-local kernel.
    
    To rewrite the original equation as a Hamiltonian system on a sequence space, we introduce an indexing scheme. Specifically, the index set is defined as \( \mathcal{Z} = \mathbb{Z}^{\mathsf{d}} \times \{-1, 1\} \). For \( J = (j, \sigma) \in \mathcal{Z} \) and a constant \( c > 0, \) we define:  
    $$|J|^2 := |j|^2 = \sum_{l=1}^{\mathsf{d}} |j_l|^2, \quad \langle j \rangle = \max\{|j|, c\}.$$  
    
    By taking the Fourier transform of \( u \) in the spatial variable \( x \), we have:  
    $$u = \sum_{k \in \mathbb{Z}} u_k \e^{\i kx}, \quad u_k = \int_{\mathbb{T}^{\mathsf{d}}} u(x) \e^{-\i kx} \d x,$$  
    where \( u_k \) are the Fourier coefficients. To simplify notation, we further denote \( u_k = u_{(k,+)} \) and \( \bar{u}_k = u_{(k,-)} \). 
    
    Using this Fourier representation, the Hamiltonian can be rewritten as:  
    \begin{equation}\label{Hamiltonian}
    	H = \sum_{k \in \mathbb{Z}} k^2 |u_k|^2 + \frac{1}{2} \sum_{k_1+k_2=k_3+k_4} K_{k_1-k_3} u_{(k_1,+)} u_{(k_2,+)} {u}_{(k_3,-)} {u}_{(k_4,-)} := H_0 + K,
    \end{equation}  
    where \( H_0 = \sum_{k \in \mathbb{Z}} k^2 |u_k|^2 \) represents the linear part, and \( K \) represents the nonlinear interaction terms.

    For monomials of degree \( d \), given by \( M=\prod_{l=1}^{d}u_{J_l} \) where \( J_l=(j_l,\sigma_l) \), we define its multi-index as \( \mathcal{J}=(J_1,\dots,J_d) \). The momentum indicator of a multi-index \( \mathcal{J} \) is given by:  
    $$\mathcal{M}_d(\mathcal{J}) = \sum_{l=1}^{d} \sigma_l j_l.$$  
    This indicator plays a role in expressing the momentum conservation condition.
    
    Our analysis will primarily focus on the monomials and polynomials whose multi-indices belong to the following set:  
    $$\mathcal{I}_d = \{\mathcal{J} \in \mathcal{Z}^d \mid \mathcal{M}_d(\mathcal{J}) = 0\},$$  
    which corresponds to the momentum conservation condition. 
    
    To quantify the regularity of solutions in the context of our stability analysis, we introduce the following weighted Banach space:  
    $$W_s = \{ u = (u_J)_{J \in \mathcal{Z}}, u_J \in \mathbb{C} \mid \Vert u \Vert_s := \sum_{j \in \mathcal{Z}} |u_j|^2 \e^{2s f(\langle j \rangle)} < \infty \}.$$  
    Here, the weight \( \e^{2s f(\langle j \rangle)} \) controls the growth of Fourier coefficients \( u_J \) as \(\langle j \rangle\) increases, and \( f \) is a monotonic function that we specify below.
    
	A function \( f \) is said to belong to the class \( \mathcal{F} \) if it satisfies the following conditions:  
	\begin{enumerate}	
		\item \( f : [c,+\infty) \to \mathbb{R}^+ \), where \( c > 0 \) is a fixed constant;  
		\item \( f \) is a monotonically increasing function that tends to \( +\infty \) as its argument grows;  
		\item There exists a constant \( C_f < 1 \) such that  
		$$f\bigg(\sum_{l=1}^d x_l\bigg) \leq f(x_m) + C_f \sum_{l \neq m} f(x_l),$$  
		where \( x_m = \max\{x_1,\dots,x_d\} \) for all \( x_l \geq c \).  
	\end{enumerate}
	This class \( \mathcal{F} \) encapsulates functions that grow sufficiently fast and satisfy a quasi-additivity property essential for controlling nonlinear interactions in stability estimates.
	
     Two typical weight functions determine spaces of infinitely differentiable functions that are not necessarily analytic, such as the Gevrey class and the logarithmic ultra-differentiable function class.  
     Specifically, for \( f(x) = x^\theta \) with \( 0 < \theta < 1 \), the associated weight corresponds to the Gevrey class function space, denoted by \( W^G_{s,\theta} \). On the other hand, for \( f(x) = (\ln x)^q \) with \( x > c \), the weight corresponds to the logarithmic ultra-differentiable function space, denoted by \( W^U_{s,q} \).  
     Lemma \ref{function} confirms that both \( x^\theta \) and \( (\ln x)^q \) satisfy the conditions of the function class \( \mathcal{F} \), making them admissible weights in this framework.
     
    We denote by \( B_s(r) \) the ball in \( W_s \) centered at the origin with radius \( r \).  
    For a functional \( H \) defined on the space \( W_s \), it induces the following Hamiltonian system:  
    $$\dot{u}_{(j,+1)} = -\i\frac{\partial H}{\partial u_{(j,-1)}},\quad 
    \dot{u}_{(j,-1)} = \i\frac{\partial H}{\partial u_{(j,+1)}}.$$  
    Here, the Hamiltonian system describes the evolution of the Fourier coefficients \( u_{(j,\pm 1)} \) under the functional \( H \), where \( \i \) represents the imaginary unit.
    
    By setting \( \overline{J} = (j, -\sigma) \) for \( J = (j, \sigma) \), the corresponding vector field is expressed as:  
    $$X_H(u) := (X_J)_{J \in \mathcal{Z}}, \quad (X_H)_{(j,\sigma)} = -\sigma \i\frac{\partial H}{\partial u_{\overline{(j,\sigma)}}}.$$  
    This vector field \( X_H(u) \) encodes the dynamics of the system governed by the Hamiltonian \( H \), where each component corresponds to the contribution of an individual Fourier mode \( u_J \).
    
    We denote by \( B_s(r) \) the ball in \( W_s \) centered at the origin with radius \( r \).  
    For a functional \( H \) defined on the space \( W_s \), it induces the following Hamiltonian system:  
    $$\dot{u}_{(j,+1)} = -\i\frac{\partial H}{\partial u_{(j,-1)}},\quad 
    \dot{u}_{(j,-1)} = \i\frac{\partial H}{\partial u_{(j,+1)}}.$$  
    By setting \( \overline{J} = (j, -\sigma) \) for \( J = (j, \sigma) \), the corresponding vector field is expressed as:  
    $$X_H(u) := (X_J)_{J \in \mathcal{Z}}, \quad (X_H)_{(j,\sigma)} = -\sigma \i\frac{\partial H}{\partial u_{\overline{(j,\sigma)}}}.$$  
    This vector field \( X_H(u) \) encodes the dynamics of the system governed by the Hamiltonian \( H \).
    
    Throughout this work, constants with numerical subscripts refer to fixed pure constants, while constants with variable subscripts depend only on these variables and do not influence the primary conclusions.  
    Note that the norm  
    $$\Vert u\Vert_s' = \sum_{j\in\mathcal{Z}} |u_j|^2 \e^{2sf(|j|)}$$  
    is equivalent to \( \Vert u\Vert_s \). For simplicity, we henceforth use \( \sum_{j\in\mathcal{Z}} |u_j|^2 \e^{2sf(|j|)} \) to denote the norm in \( W_s \).

    \section{Resonant Normal Form}\label{sec:res}
    \subsection{Polynomials Setting}\label{sec:pol}
    A homogeneous polynomial \( P \) of degree \( d \) can be expressed in the following form:  
    \begin{equation}\label{P}  
    	P(u) = \sum_{J_1,\dots,J_d\in\mathcal{Z}}P_{J_1,\dots, J_d}u_{J_1}\dots u_{J_d},  
    \end{equation}  
    where \( P_{J_1,\dots,J_d} \) are the coefficients and \( u_{J_l} \) represent the variables associated with the index \( J_l \).
      
    By introducing the multi-index notation \( \mathcal{J} = (J_1, \dots, J_d) \), this expression can be rewritten as  
    \[  
    P(u) = \sum_{\mathcal{J}\in\mathcal{Z}^d}P_{\mathcal{J}}u^\mathcal{J},  
    \]  
    where \( P_{\mathcal{J}} \) and \( u^\mathcal{J} = u_{J_1} \dots u_{J_d} \) simplify the notation for brevity.  
    Based on this framework, we proceed to define the class of polynomials under consideration.

    \begin{definition}
    	Let \( d \geq 1 \). We define \( \mathcal{P}_d \) as the space of formal homogeneous polynomials \( P(u) \) of degree \( d \), expressed in the form \eqref{P}, and satisfying the following conditions:  
    	\begin{enumerate}  
    		\item \textbf{Momentum Conservation}:  
    		\( P(u) \) is supported on monomials with zero momentum indicator, i.e.,  
    		\[  
    		P(u) = \sum_{\mathcal{J} \in \mathcal{I}_d} P_{\mathcal{J}} u_{J_1} \dots u_{J_d},  
    		\]  
    		where \( \mathcal{I}_d \) denotes the set of multi-indices satisfying \( \mathcal{M}_d(\mathcal{J}) = 0 \);  
    		
    		\item \textbf{Reality}:  
    		For any \( \mathcal{J} \in \mathcal{Z}^d \), we require  
    		\[  
    		\overline{P_{\mathcal{J}}} = P_{\overline{\mathcal{J}}},  
    		\]  
    		ensuring that \( P(u) \) represents a real-valued functional in the context of the system;  
    		
    		\item \textbf{Boundedness}:  
    		The coefficients \( P_{\mathcal{J}} \) are uniformly bounded, i.e.,  
    		\[  
    		C_P := \sup_{\mathcal{J} \in \mathcal{I}_d} |P_{\mathcal{J}}| < \infty.  
    		\]  
    	\end{enumerate}
      \end{definition}
    For given \( r, s > 0 \), the space \( \mathcal{P}_d \) is equipped with the norm:  
    \[  
    |P(u)|_{r,s} := \frac{1}{r} \sup_{u \in B_s(r)} \Vert X_{\underline{P}}(u) \Vert_s,  
    \]  
    where  
    \[  
    \underline{P}(u) = \sum_{\mathcal{J} \in \mathcal{I}_d} |P_{\mathcal{J}}| u_{J_1} \dots u_{J_d}.  
    \]  
    
    For integers \( \infty > d_2 \geq d_1 \geq 1 \), the space \( \mathcal{P}_{d_1,d_2} := \bigcup_{k=d_1}^{d_2} \mathcal{P}_k \) consists of polynomials  
    \[  
    P = \sum_{k=d_1}^{d_2} P_k, \quad P_k \in \mathcal{P}_k,  
    \]  
    and inherits the same norm given by  
    \[  
    |P|_{r,s} := \frac{1}{r} \sup_{u \in B_s(r)} \Vert X_{\underline{P}} \Vert_s.  
    \]   Similarly, we define \( \mathcal{P}_{d,\infty} := \bigcup_{k \geq d} \mathcal{P}_k \), and any \( P \in \mathcal{P}_{d,\infty} \) takes the form  
    \[  
    P = \sum_{k \geq d} P_k, \quad P_k \in \mathcal{P}_k,  
    \]  
    with the norm as above. When \( d_1 > d_2 \), we define \( \mathcal{P}_{d_1,d_2} := \emptyset. \)
    
    For \( P_1, P_2 \in \mathcal{P}_{d_1,d_2} \), the Poisson bracket is defined as:  
    \[  
    \{P_1, P_2\} := -\i \sum_{(j,\sigma) \in \mathcal{Z}} \sigma \frac{\partial P_1}{\partial u_{(j,\sigma)}} \frac{\partial P_2}{\partial u_{(j,-\sigma)}}.  
    \]  
    
    In this section, we are particularly interested in the following quantity associated with the multi-index \( \mathcal{J} = (J_1, \dots, J_d) \):  
    \[
    \mathcal{E}_d(\mathcal{J}) = \sum_{l=1}^{d} \sigma_l j_l^2,  
    \]  
    which characterizes whether the index \( \mathcal{J} \) resonates with the quadratic terms \( H_2 \) in the Hamiltonian. Based on this criterion, the resonant set is defined as:  
    \[
    \mathcal{R}_d = \{\mathcal{J} \colon \mathcal{E}_d(\mathcal{J}) = 0\},  
    \]  
    and its complement in \( \mathcal{Z}^d \) is given by:  
    \[
    \mathcal{N}_d = \mathcal{Z}^d \setminus \mathcal{R}_d.  
    \]  
    We also define the non-resonant set over all degrees as:  
    \[
    \mathcal{N} = \bigcup_d \mathcal{N}_d.  
    \]

    \subsection{Resonant Normal Form Lemma}\label{sec:nfl}
    \begin{proposition}\label{resonantlize}
    	Let \( d > 3 \), \( s > s_0 \), and suppose the radius \( r > 0 \) satisfies the constraint:
    	\[
    	3(C_K + 1) r d^2 < 1.
    	\]  
    	For the Hamiltonian \( H \) in \eqref{Hamiltonian} defined on the ball \( B_s\left(\frac{r}{2}\right) \), there exists a symplectic map \( \Phi: B_s\left(\frac{r}{2}\right) \to B_s(r) \) with the following properties:
    	\begin{enumerate}  
    		\item
    		\[
    		H \circ \Phi^{-1} = H_0 + Z_d + R_d,  
    		\]  
    		where \( H_0 \) is the unperturbed quadratic Hamiltonian, \( Z_d \) represents the resonant terms, and \( R_d \) stands for the non-resonant remainder.  
    		
    		\item  
    		The resonant terms \( Z_d \) belong to \( \mathcal{P}_{4,2d+2} \), satisfy the resonance condition \( \{H_0, Z_d\} = 0 \), and are uniformly bounded as:  
    		\[  
    		|Z_d|_{r,s} \leq 2C_K r^2.  
    		\]  
    		
    		\item
    		The symplectic map \( \Phi \) approximates the identity on \( B_s\left(\frac{r}{2}\right) \):  
    		\[  
    		\sup_{u \in B_s\left(\frac{r}{2}\right)} \|\Phi(u) - u\|_s \leq 2C_K r^3,  
    		\]  
    		which further implies:  
    		\[  
    		\|\Phi(u) - u\|_s \leq 16C_K \|u\|_s^3, \quad \text{for all } u \in B_s\left(\frac{r}{2}\right).  
    		\]
    		
    		\item 
    		The remainder \( R_d \) belongs to \( \mathcal{P}_{2d+2,\infty} \) and satisfies the bound:  
    		\[
    		|R_d|_{r,s} \leq 6C_K^d r^{2d-2} d^{5d+1}.  
    		\]
    	\end{enumerate}  
    \end{proposition}

    \begin{proof}
    	We prove this proposition by induction. Let
    	\[
    	r_k = r - \frac{k-1}{2d-2}r.
    	\]
    	At each step \(1 \leq k \leq d\), we will show that there exists a transformation \(\Phi_k\) satisfying the following properties:
    	\begin{enumerate}
    		\item \(H_k = H \circ \Phi_k^{-1} = H_0 + Z_k + P_k + R_k\);
    		\item \(P_k \in \mathcal{P}_{2k+2,2d}\) and \(|P_k|_{r_k,s} \leq C_K r^{2k} d^{2k-2}\);
    		\item \(\sup_{u \in B_s(r_k)} \|\Phi_k(u) - u\|_s \leq \sum_{l=1}^k C_K r^{2l+1} d^{2l-2}\);
    		\item \(Z_k \in \mathcal{P}_{2k+2,2d}, \, \{Z_k, H_0\} = 0, \, |Z_k|_{r_k,s} \leq 2C_K r^2\);
    		\item \(R_k \in \mathcal{P}_{2k+2,\infty}, \, |R_k|_{r_k,s} \leq (d-2) \left(\left(\frac{d-1}{d-2}\right)^{k-1} - 1\right) 3C_K^d r^{2d-2} d^{5d}\).
    	\end{enumerate}
    	
    	Suppose that the conditions of the proposition are satisfied for some $k \geq 1$.
        We will prove that they also hold for $k+1$. 
        To construct the next symplectic transformation, we solve the following homological equation:
        \begin{equation*}
        	\{H_0, S_k\} + P_k = \Delta Z_k.
        \end{equation*}
        
       Given the expansion 
       $$P_k = \sum_{l=k+1}^d \sum_{\mathcal{J} \in \mathcal{I}_l} P_{2l, \mathcal{J}} u_{\mathcal{J}},$$ 
       we can solve the homological equation by comparing coefficients:
       $$
       S_k = \sum_{l=k+1}^d \sum_{\mathcal{J} \in \mathcal{I}_l \setminus \mathcal{R}_l} 
       \frac{P_{2l,\mathcal{J}}}{\i\mathcal{E}_l(\mathcal{J})} u_{\mathcal{J}}, 
       \quad 
       \Delta Z_k = \sum_{l=k+1}^d \sum_{\mathcal{J} \in \mathcal{R}_l} P_{2l,\mathcal{J}} u_{\mathcal{J}}.
       $$
       Consequently, we have 
       $$|S_k|_{r_k,s} \leq |P_k|_{r_k,s} \leq \frac{1}{2d} \leq \delta_k := \frac{r_{k} - r_{k+1}}{8\e r_k}.$$
       This follows from the fact that $|\mathcal{E}_l(\mathcal{J})| \geq 1$ for $\mathcal{J} \in \mathcal{I}_l \setminus \mathcal{R}_l$ and the assumptions imposed on $r$.

    Let $\phi_k^t$ denote the Hamiltonian flow generated by $S_k$. We define the transformation $\phi_k := \phi_k^1$ with its inverse $\phi_k^{-1}$.

    We now estimate the near-identity property of $\phi_k^{\pm}$:
    $$
    \sup_{u \in B_s(r_{k+1})} \|\phi_k^{\pm}(u) - u\|_s 
    \leq \sup_{u \in B_s(r_{k+1})} \|X_{S_k}(u)\|_s 
    \leq r_{k+1} |S_k|_{r_{k+1},s} 
    \leq C_K r^{2k+1} d^{2k-2} 
    \leq r_k - r_{k+1},
    $$
    which implies that $\phi_k^{\pm}: B_s(r_{k+1}) \to B_s(r_k)$.

    We define the transformation for step $k+1$ by composing $\phi_k$ with $\Phi_k$, i.e.,
    $$
    \Phi_{k+1} = \Phi_k \circ \phi_k.
    $$
    Using the telescope argument, we estimate the near-identity property of $\Phi_k$.
   
    	\begin{align*}
    		\sup_{u\in B_{s}(r_{k+1})}\Vert\Phi_{k+1}(u)-u\Vert_s&\leq  \sup_{u\in B_{s}(r_{k+1})}\Vert\Phi_{k}\circ\phi_k(u)-\phi_k (u)\Vert_s+\Vert\phi_k(u)-u\Vert_s \\
    		&\leq  \sup_{u\in B_{s}(r_{k})}\Vert\Phi_{k}(u)-u\Vert_s+\sup_{u\in B_{s}(r_{k+1})}\Vert\phi_k(u)-u\Vert_s\\
    		&\leq\sum_{l=1}^{k-1}C_Kr^{2l+1}d^{2l-2}+C_Kr^{2k+1}d^{2k-2}\\
    		&\leq\sum_{l=1}^{k}C_Kr^{2l+1}d^{2l-2}.
    	\end{align*}
    	
    	Due to the group property of the generating flow, $\phi^t$ is locally invertible. Consequently, the new Hamiltonian, defined on $B_s(r_{k+1})$, can be expressed as:
    	\[
    	H_{k+1} = H_k \circ \phi_k^{-1}.
    	\]

    	Using Taylor's formula with integral remainders, the transformed Hamiltonian $H_k \circ \phi_k^{-1}$ can be expressed in the following form:

    	 \begin{align*}
    	 	H_k\circ\phi_k^{-1}&=H_0+\{H_0,S_k\}+\sum_{l=2}^{n_1}\frac{ad_{S_k}^l}{l!}H_0+R_{H_0,k}\\
    	 	&+Z_k+\sum_{l=1}^{n_2}\frac{ad_{S_k}^l}{l!}Z_k+R_{Z_k,k}\\
    	 	&+P_k+\sum_{l=1}^{n_3}\frac{ad_{S_k}^l}{l!}P_k+R_{P_k,k}+R_k\circ\phi_k^{-1},
    	 \end{align*}
    	 where
    	 	\begin{align*}
    	 	R_{H_0,k}&=\int_{0}^{1}\frac{(1-\tau)^{n_1}}{n_1!}ad_{S_k}^{n_1+1}(H_0)\circ\phi^{-\tau}(u)\d \tau,\\
    	 	R_{Z_k,k}&=\int_{0}^{1}\frac{(1-\tau)^{n_2}}{n_2!}ad_{S_k}^{n_2+1}(Z_k)\circ\phi^{-\tau}(u)\d \tau,\\
    	 	R_{P_k,k}&=\int_{0}^{1}\frac{(1-\tau)^{n_3}}{n_3!}ad_{S_k}^{n_3+1}(P_k)\circ\phi^{-\tau}(u)\d \tau.
    	 \end{align*}
        The integers \(n_1, n_2, n_3\) are selected such that the degree of the terms in the integral remainders exceeds \(2d\). The specific values of \(n_1, n_2, n_3\) will be determined in subsequent remainder estimates.
        
    	 We define:
    	 $$
    	 P_{k+1} := \sum_{l=2}^{n_1} \frac{\mathrm{ad}_{S_k}^l}{l!} H_0 + \sum_{l=1}^{n_2} \frac{\mathrm{ad}_{S_k}^l}{l!} Z_k + \sum_{l=1}^{n_3} \frac{\mathrm{ad}_{S_k}^l}{l!} P_k,
    	 \quad
    	 Z_{k+1} := Z_k + \Delta Z_k,
    	 $$
    	 $$
    	 R_{k+1} := R_{H_0, k} + R_{Z_k, k} + R_{P_k, k} + R_k \circ \phi_k^{-1}.
    	 $$
    	 
    	Using the homological equation, we conclude that the map $\phi_k$ transforms $H_k$ into the desired form of $H_{k+1}$. We now proceed with the necessary remainder estimates.
    	
    	Setting \(\delta_k = \frac{r_{k+1}}{r_k}\), we verify that 
    	\[
    	|S_k|_{r_k, s} \leq C_K r^{2k} d^{2k-2} \leq \frac{1}{2} \leq \delta_k \leq 1.
    	\]
    	We now estimate the components of \(P_{k+1}\) using Lemma \ref{Lie} and the assumptions:
    	\begin{align*}
    		\left| \sum_{l=2}^{n_1} \frac{\mathrm{ad}_{S_k}^l}{l!} H_0 \right|_{r_{k+1}, s}
    		&= \left| \sum_{l=1}^{n_1-1} \frac{\mathrm{ad}_{S_k}^l}{(l+1)!} \{S_k, H_0\} \right|_{r_{k+1}, s} \\
    		&\leq \sum_{l=1}^{n_1-1} \frac{1}{(l+1)!} \left(\frac{|S_k|_{r_k, s}}{2\delta_k} \right)^l |P_k|_{r_k, s} \\
    		&\leq \left(\exp\left(\frac{|S_k|_{r_k, s}}{2\delta_k}\right) - 1 \right)\frac{|S_k|_{r_k, s}}{2\delta_k} |P_k|_{r_k, s} \\
    		&\leq C_K^2 r^{4k} d^{4k-3} \\
    		&\leq \frac{1}{3} C_K r^{2k+2} d^{2k},
    	\end{align*}
    	
    	\begin{align*}
    		\left| \sum_{l=1}^{n_2} \frac{\mathrm{ad}_{S_k}^l}{l!} Z_k \right|_{r_{k+1}, s}
    		&\leq \sum_{l=1}^{n_2} \frac{1}{l!} \left(\frac{|S_k|_{r_k, s}}{2 \delta_k} \right)^l |Z_k|_{r_k, s} \\
    		&\leq \left(\exp\left(\frac{|S_k|_{r_k, s}}{2\delta_k}\right) - 1 \right) \frac{|S_k|_{r_k, s}}{2\delta_k} |Z_k|_{r_k, s} \\
    		&\leq 2C_K^2 r^{2k+2} d^{2k-1} \\
    		&\leq \frac{1}{3} C_K r^{2k+2} d^{2k},
    	\end{align*}
    	
    	\begin{align*}
    		\left| \sum_{l=1}^{n_3} \frac{\mathrm{ad}_{S_k}^l}{l!} P_k \right|_{r_{k+1}, s}
    		&\leq \sum_{l=1}^{n_3} \frac{1}{l!} \left(\frac{|S_k|_{r_k, s}}{2 \delta_k} \right)^l |P_k|_{r_k, s} \\
    		&\leq \left(\exp\left(\frac{|S_k|_{r_k, s}}{2 \delta_k}\right) - 1 \right) \frac{|S_k|_{r_k, s}}{2\delta_k} |P_k|_{r_k, s} \\
    		&\leq C_K^2 r^{4k} d^{2k-3} \\
    		&\leq \frac{1}{3} C_K r^{2k+2} d^{2k}.
    	\end{align*}
    	
    	Combining the three estimates above, we establish the bound for \(P_{k+1}\).
    	Next, we estimate \(Z_{k+1}\):
    	\begin{align*}
    		|Z_{k+1}|_{r_{k+1}, s}    
    		&\leq \left|\sum_{l=1}^{k} Z_l\right|_{r_{k+1}, s} \\
    		&\leq \sum_{l=1}^{k} |\Delta Z_l|_{r_l, s} \\
    		&\leq \sum_{l=1}^{k} |P_l|_{r_l, s} \\
    		&\leq \frac{C_K r^2}{1 - r^2 d^3} \\
    		&\leq 2C_K r^2.
    	\end{align*}

      To estimate \(R_{H_0, k}\), we take \(\frac{d}{k} - 1 \leq n_1 = \left\lfloor \frac{d}{k} \right\rfloor \leq \frac{d}{k}\), then
      \begin{align*}
      	|R_{H_0, k}|_{r_{k+1}, s} 
      	&\leq \left(\frac{|S_{k}|_{r_k, s}}{2\delta_k}\right)^{n_1} |\{S_k, H_0\}| \\
      	&\leq d^{n_1} \left(C_K r^{2k} d^{2k-2}\right)^{n_1+1} \\
      	&\leq C_K^d r^{2d} d^{4d}.
      \end{align*}
      
       For \(R_{Z_k, k}\), we take \(n_2 = n_1 = \left\lfloor \frac{d}{k} \right\rfloor\), then
       \begin{align*}
       	|R_{Z_k, k}|_{r_{k+1}, s} 
       	&\leq \left(\frac{|S_{k}|_{r_k, s}}{2\delta_k}\right)^{n_2 + 1} |Z_k| \\
       	&\leq d^{n_2+1} \left(C_K r^{2k} d^{2k-2}\right)^{n_2+1} 2C_K r^2 \\
       	&\leq C_K^d r^{2d} d^{4d}.
       \end{align*}
       
       For \(R_{P_k, k}\), we take \(n_3 = \left\lfloor \frac{d-2}{k} \right\rfloor - 1\), then
       \begin{align*}
       	|R_{P_k, k}|_{r_{k+1}, s} 
       	&\leq \left(\frac{|S_{k}|_{r_k, s}}{2 \delta_k}\right)^{n_3 + 1} |P_k| \\
       	&\leq d^{n_3+1} \left(C_K r^{2k} d^{2k-2}\right)^{n_3 + 2} \\
       	&\leq C_K^d r^{2d-2} d^{5d}.
       \end{align*}
       
       Thus, we can estimate the remainder \(R_{k+1}\) as follows:
       \begin{align*}
       	|R_{k+1}|_{r_{k+1}, s} 
       	&\leq |R_k \circ \phi_k^{-1}|_{r_{k+1}, s} 
       	+ |R_{H_0, k}|_{r_{k+1}, s} 
       	+ |R_{Z_k, k}|_{r_{k+1}, s} 
       	+ |R_{P_k, k}|_{r_{k+1}, s} \\
       	&\leq \frac{r_k}{r_{k+1}} |R_k|_{r_k, s} 
       	+ 2 C_K^d r^{2d} d^{4d} 
       	+ C_K^d C^{2d} r^{2d-2} d^{5d} \\
       	&\leq \frac{d-1}{d-2}|R_k|_{r_k, s} + 3 C_K^d r^{2d-2} d^{5d} \\
       	&\leq (d-2) \left(\left(\frac{d-1}{d-2}\right)^k - 1\right) 3 C_K^d r^{2d-2} d^{5d}.
       \end{align*}
    
    \end{proof}
    
    Furthermore, we examine the indices \(\mathcal{J} = \{(k_1, +), (k_2, +), (k_3, -), (k_4, -)\}\) corresponding to the term \(\overline{K}_2 := \Delta Z_1\). These indices must satisfy
    \[
    k_1 + k_2 = k_3 + k_4, \quad k_1^2 + k_2^2 = k_3^2 + k_4^2,
    \]
    which admits only trivial solutions of the form \(\{k_1, k_2\} = \{k_3, k_4\}\). This shows that
    \[
    \overline{K}_2 = \sum_{k_1, k_2 \in \mathbb{Z}} K_{|k_1 - k_2|} |u_1|^2 |u_2|^2
    \]
    is also an integrable term. This property will be crucial in the subsequent construction of the integrable normal form.
    
    \section{Truncation Estimate}\label{sec:tru}
    
    For an integer \(M\), we partition the index set \(\mathcal{Z}\) into two disjoint subsets: the low-mode indices \(\{|J| \leq M\}\) and the high-mode indices \(\{|J| > M\}\). Accordingly, \(u \in W_s\) can be decomposed into its low- and high-mode components:
    \[
    u = u^{>M} + u^{<M} := \sum_{|J| > M} u_J + \sum_{|J| \leq M} u_J.
    \]
    We classify polynomials based on the order of vanishing in the high-mode variables \(u^{>M}\) in this section. 
    
    \begin{lemma}\label{13}
    	Let \(d \geq 4\), and let \(\mathcal{J} = (J_1, J_2, \dots, J_d) \in \mathcal{R}_d\) be an ordered set of indices such that \(|J_1| \geq |J_2| \geq \dots \geq |J_d|\). If \(|J_1| \geq M\) and \(J_1 \neq \overline{J_2}\), then \(|J_3| \geq \sqrt{\frac{M}{d-2}}\). 
    \end{lemma}
    \begin{proof}
    	If \(|J_1| = |J_2|\), the resonance condition for \(\mathcal{J}\) implies that \(\sigma_1 = \sigma_2\). It follows that:
    	\[
    	M^2 \leq |J_1|^2 + |J_2|^2 = -\sum_{l=3}^{d} \sigma_l |J_l|^2 \leq (d-2) |J_3|^2.
    	\]
    	
    	If \(|J_1| > |J_2|\), the resonance condition implies that
    	\[
    	M + 1 \leq |J_1| + |J_2| \leq |J_1|^2 - |J_2|^2 = -\sum_{l=3}^{d} \sigma_l |J_l|^2 \leq (d-2) |J_3|^2.
    	\]
    	
    	In either case, the conclusion holds.
    \end{proof}
     
     \begin{lemma}
     	Let \(J^* \in \mathcal{Z}\) be a fixed index with \(|J^*| > M\), and let \(P = P_{\mathcal{J}} u_{\mathcal{J}}\) be a monomial, where \(\mathcal{J} = \{J_1, \dots, J_d\} \in \mathcal{R}\). Then the Poisson bracket 
     	\(\{|u_{J^*}|^2, P\}\) vanishes to at least order 3 in the high-mode variable \(u^{>N}\), where \(N \geq \sqrt{\frac{M}{d-2}}\).
     \end{lemma}
     
     \begin{proof}
     	We consider three cases:
     	
     	1. If \(|J^*| > |J_1|\), then the indices of \(P\) are disjoint from \(\{J^*, \overline{J^*}\}\), which means \(\{|u_{J^*}|^2, P\} = 0\). The conclusion holds trivially.
     	
     	2. If \(|J^*| = |J_1| = |\overline{J_2}|\), then \(\{|u_{J^*}|^2, P\} = 0\), and the conclusion again holds.
     	
     	3. If \(N < |J^*| \leq |J_1|\) and \(J_1 \neq \overline{J_2}\), then the Poisson bracket is non-zero. Any resulting monomial contains indices that must satisfy the conditions of Lemma \ref{13}. Therefore, the third largest index, \(J_3'\), in any monomial of \(\{|u_{J^*}|^2, P\}\) must satisfy \(|J_3'| \geq \sqrt{\frac{M}{d-2}}\). This ensures that at least three indices in the monomial correspond to high modes (modes \(>N\)), which implies that the polynomial vanishes to order 3 in \(u^{>N}\).
     \end{proof}

     \begin{lemma}[Truncation Estimate]\label{cut}
     	Let \(s > s_0\), and let 
     	\(P \in \mathcal{P}_d\) be a polynomial of degree \(d \geq 3\) that vanishes to at least order 3 in the high-mode variables \(u^{>} = (u_{J})_{|J| > N}\). Then its norm is bounded as follows:
     	\[
     	|P|_{r,s} \leq C_{P} \frac{(2r)^{d-2}}{\mathrm{e}^{(s-s_0)f(N)}}.
     	\] 
     \end{lemma}

    \begin{proof}
    	The condition that \(P\) vanishes to at least order 3 in \(u^{>N}\) allows us to write it in the form 
    	\[
    	P(u) = \sum_{l=3}^{d} P_l(u^{>N}, u^{<N}),
    	\] 
    	where each \(P_l\) is a polynomial that is homogeneous of degree \(l\) in \(u^{>N}\) and degree \(d-l\) in \(u^{<N}\).
    	
    	For \(|j| \leq N\), the associated component of the vector field is given by
    	\[
    	(X_P(u^{>}, u^{<}))_j = \sum_{l=3}^{d} \partial_{u_{\bar{j}}} P_l(u^{>N}, u^{<N}),
    	\]
    	which is a polynomial that is homogeneous of degree at least 3 in \(u^{>N}\). For \(|j| > N\), similarly, we have:
    	\[
    	(X_P(u^{>}, u^{<}))_j = \sum_{l=3}^{d} \partial_{u_{\bar{j}}} P_l(u^{>N}, u^{<N}),
    	\]
    	which is homogeneous of degree at least 2 in \(u^{>N}\).
    	
    	Using a technical result (Lemma A.1 in the Appendix of \cite{BFM24}), we can bound the norm of the vector field as follows:
    	\[
    	\Vert (X_P)(u^{>}, u^{<}) \Vert_s 
    	\leq C_{P} 2^d \big(\Vert u^{>} \Vert_{s_0} \Vert u^{<} \Vert_{s}^{d-2} + \Vert u^{>} \Vert_{s_0}^2 \Vert u^{<} \Vert_s^{d-3}\big).
    	\]
    	
    	Furthermore, the norm of the high-mode component \(u^{>}\) can be estimated as:
    	\[
    	\Vert u^{>} \Vert_{s_0}^2 
    	= \sum_{|J| > N} \mathrm{e}^{2s_0 f(|J|)} |u_J|^2 
    	= \sum_{|J| > N} \frac{\mathrm{e}^{2sf(|J|)} |u_J|^2}{\mathrm{e}^{(2s - 2s_0) f(|J|)}} 
    	\leq \frac{\Vert u \Vert_s^2}{\mathrm{e}^{(2s - 2s_0)f(N)}}.
    	\]
    	
    	Combining these two inequalities, we arrive at the bound:
    	\[
    	\sup_{u \in B^s(r)} \Vert (X_P)(u^{>}, u^{<}) \Vert_s 
    	\leq C_{P} \frac{2^d r^{d-1}}{\mathrm{e}^{(s-s_0)f(N)}}.
    	\]
    	
    	By the definition of the norm \( |P|_{r,s} \), this implies the desired result:
    	\[
    	|P|_{r,s} \leq C_{P} \frac{(2r)^{d-2}}{\mathrm{e}^{(s-s_0)f(N)}}.
    	\]
    \end{proof}

    \section{Integrable Normal Form}\label{sec:int}
    
    \subsection{Frequencies and the Rational Normal Form}\label{sec:fre}
    
    Based on the low and high mode decomposition, we define the following projection operators:
    \[
    \Pi^{>M}(u) := u^{>M}, \quad \Pi^{<M}(u) := u^{<M}, \quad \text{for any } u \in W_s.
    \]
    In this section, we construct a rational normal form on the finite-dimensional space 
    \[
    W_s^M := \{u \mid \Pi^{>M}(u) = 0\}, 
    \]
    restricted to the ball \(B_s^M(r) := B_s(r) \cap W_s^M\).
    
    We denote the set of low-mode truncated indices as:
    \[
    \mathfrak{J}^{d,M} := \Big\{\mathcal{J} = \{J_1, \dots, J_d\} \mid |J_i| \leq M, \, 1 \leq i \leq d \Big\}, 
    \quad \mathfrak{J}^M := \bigcup_{d \geq 2} \mathfrak{J}^{d,M}.
    \]   For any \(\mathcal{J} \in \mathfrak{J}^{d,M}\), we define the corresponding frequency as:
    \[
    \omega^M_{\mathcal{J}}(u) := \mathrm{i} \sum_{\alpha=1}^d \frac{\partial K_2(u)}{\partial |u_{J_\alpha}|^2} 
    = 2 \mathrm{i} \sum_{\alpha=1}^d \sum_{k \in \mathbb{Z}} K_{|k-\alpha|} |u_k|^2.
    \]

    We now demonstrate that the frequency \(\omega^M_{\mathcal{J}}\) is Lipschitz continuous with respect to \(u\).
    
    \begin{lemma}[Lipschitz Property of the Frequencies]\label{lip}
    	For \(u, u' \in W_s\) and \(\mathcal{J} \in \mathfrak{J}^{d,M}\), the following estimate holds:
    	\[
    	|\omega_{\mathcal{J}}^M(u) - \omega_{\mathcal{J}}^M(u')| 
    	\leq 2d \Vert u - u' \Vert_s (\Vert u \Vert_s + \Vert u' \Vert_s).
    	\]
    \end{lemma}
    \begin{proof}
    	Note that:
    	\begin{align*}   
    		|\omega_{\mathcal{J}}^M(u) - \omega_{\mathcal{J}}^M(u')| 
    		&\leq 2 \sum_{\alpha=1}^{d} \sum_{k \in \mathbb{Z}} K_{|k - J_\alpha|} \big| |u_k|^2 - |u'_k|^2 \big| \\
    		&\leq 2 \sum_{\alpha=1}^{d} \sum_{k \in \mathbb{Z}} K_{|k - J_\alpha|} \mathrm{e}^{-2s f(|k|)} \mathrm{e}^{2s f(|k|)} \big| |u_k|^2 - |u'_k|^2 \big| \\
    		&\leq 2 \sum_{\alpha=1}^{d} \sum_{k \in \mathbb{Z}} K_{|k - J_\alpha|} \mathrm{e}^{-2s f(|k|)} \mathrm{e}^{2s f(|k|)} \big| |u_k - u'_k| \cdot (|u_k| + |u'_k|) \big| \\
    		&\leq 2d \Vert u - u' \Vert_s (\Vert u \Vert_s + \Vert u' \Vert_s).
    	\end{align*}
    \end{proof}
    
    We now define several non-resonant domains, where the rational normal form will be constructed:
    \[
    \mathfrak{D}_{\gamma}^{d,M} := \Big\{ u \in W_s^M \,\Big|\, \min_{\mathcal{J} \in \mathfrak{J}^{d,M}} \big| \omega^M_{\mathcal{J}}(u) \big| > \gamma \Vert u \Vert_s^2 \Big\},
    \]
    \[
    \mathfrak{D}_{0^+}^{d,M} := \Big\{ u \in W_s^M \,\Big|\, \min_{\mathcal{J} \in \mathfrak{J}^{d,M}} \big| \omega^M_{\mathcal{J}}(u) \big| > 0 \Big\},
    \]
    \[
    \mathfrak{B}_{\gamma, s}^{d,M}(r) := B_s^M(r) \cap \mathfrak{D}_{\gamma}^{d,M}.
    \]
    
   \begin{definition}  
   	We define the rational normal form \(Q(u)\) as a formal series:  
   	\begin{equation}\label{fraction}  
   		Q(u) = \sum_{\mathcal{J} \in \mathcal{R} \cap \mathfrak{J}^M} f_{\mathcal{J},h}(u) u_{\mathcal{J}}  
   		= \sum_{\mathcal{J} \in \mathcal{R} \cap \mathfrak{J}^M} \sum_{h \in \overline{\mathcal{N}}} Q_{\mathcal{J},h}  
   		\prod_{\alpha=1}^{\#h} \frac{\mathrm{i}}{\omega_{h_\alpha}^M(u)},  
   	\end{equation}  
   	where \(\overline{\mathcal{N}} = \bigcup_{n \geq 1} \big(\mathcal{N} \cap \mathfrak{J}^M\big)^n\).  
   	
   	We use the following definitions to characterize the structure of $Q(u)$:
    	
    \begin{enumerate}  
    	\item Order \(q\):
    	For all \(\mathcal{J}, h\) satisfying \(Q_{\mathcal{J},h} \neq 0\), the order \(q\) is defined as:  
    	\[
    	q = \#\mathcal{J} - 2\#h \geq 4.
    	\]
    	\item The maximum number of terms in a simple multiplier of the denominator, denoted by \(\mathfrak{h}_Q\): 
    	\[
    	\mathfrak{h}_Q := \sup_{Q_{\mathcal{J},h} \neq 0} \sup_{1 \leq \alpha \leq \#h} \#h_\alpha.
    	\]
    	\end{enumerate}
     We denote \(Q \in \mathcal{H}^M_q\) if \(Q\) admits the expansion given in \eqref{fraction},  
      and satisfies the index conditions defined above.  
 	
    \end{definition}

  \begin{remark}
  	A monomial \(u_{\mathcal{J}}\) with \(\mathcal{J} \in \mathcal{R}\) can be regarded as a trivial rational normal form, where \(\mathfrak{h}_Q = 0\) and \(\mathfrak{n}_Q = 0\).
  \end{remark}
  
  For a rational normal form \(Q(u) \in \mathcal{H}^M_q\) with generic degree \(q\), defined on \(\mathfrak{B}_{\gamma,s}^{\mathfrak{h}_Q,M}(r)\), we define its norm as:  
  \[
  |Q|_{[r,s,\gamma]} := \sup_{u \in \mathfrak{B}_{\gamma,s}^{\mathfrak{h}_Q,M}(r)} \frac{\Vert X_Q(u) \Vert_s}{\Vert u \Vert_s}.
  \]
  
  Since the generic degree \(q\) of the rational normal form \(Q\) satisfies \(q \geq 4\), we have  
  \[
  \lim\limits_{\Vert u \Vert_s \to 0} \frac{\Vert X_Q(u) \Vert_s}{\Vert u \Vert_s} = 0.
  \]  
  Thus, the norm \(|Q|_{[r,s,\gamma]}\) is well defined.
  
   It follows directly from the definition that if \(r > r'\) and \(\gamma < \gamma'\), then:  
   \[
   |\cdot|_{[r,s,\gamma]} \geq |\cdot|_{[r',s,\gamma]}, \quad |\cdot|_{[r,s,\gamma]} \geq |\cdot|_{[r,s,\gamma']}.
   \]

    Now we can establish some estimates for the rational normal form.
    \begin{lemma}[Cauchy estimate]\label{Cauchy}  
    	Let \(Q(u) \in \mathcal{H}^M_q\) be defined on \(\mathfrak{B}_{\gamma,s}^{\mathfrak{h}_Q,M}(r)\), and assume \(X_Q(u) \in W_s\).  
    	Then the differential \(DX_Q(u)\) belongs to \(\mathcal{L}(W_s, W_s)\).  
    	
    	Moreover, for any \(\rho\) such that  
    	\[
    	\gamma' = \left(\frac{1}{1+\rho}\right)^2 \gamma - 2 \mathfrak{h}_Q \frac{\rho}{1-\rho} > 0,
    	\]  
    	the following estimate holds for \(u \in \mathfrak{D}_{\gamma}^{\mathfrak{h}_Q,M}\):  
    	\[
    	\sup_{\Vert h \Vert_s = 1} \Vert DX_Q(u) h \Vert_s \leq \frac{1}{\rho \Vert u \Vert_s} \sup_{\Vert u' - u \Vert_s = \rho \Vert u \Vert_s} \Vert X_Q(u') \Vert_s.
    	\]
    \end{lemma}
    
    \begin{proof}  
    	We apply Lemma \ref{lip} for \(u \in \mathfrak{D}_{\gamma}^{\mathfrak{h}_Q,M}\) to obtain:  
    	\begin{align*}  
    		\inf_{\Vert u' - u \Vert_s = \rho \Vert u \Vert_s} \big|\omega_{\mathcal{J}}^M(u')\big|  
    		&= \inf_{\Vert u' - u \Vert_s = \rho \Vert u \Vert_s} \big|\omega_{\mathcal{J}}^M(u') - \omega_{\mathcal{J}}^M(u) + \omega_{\mathcal{J}}^M(u)\big| \\  
    		&\geq \inf_{u \in \mathfrak{D}_{\gamma}^{\mathfrak{h}_Q,M}} \big|\omega_{\mathcal{J}}^M(u)\big|  
    		- \sup_{\Vert u - u' \Vert_s = \rho \Vert u \Vert_s} \big|\omega_{\mathcal{J}}^M(u) - \omega_{\mathcal{J}}^M(u')\big| \\  
    		&\geq \gamma \Vert u \Vert_s^2 - 2 \rho \mathfrak{h}_Q \Vert u' \Vert_s \Vert u \Vert_s \\  
    		&\geq \gamma' \Vert u' \Vert_s^2.  
    	\end{align*}  	This implies that, with this choice of \(\rho\), the disk of radius \(\rho \Vert u \Vert_s\),  
    	centered at any point \(u\) within \(\mathfrak{D}_{\gamma}^{\mathfrak{h}_Q,M}\), contains no points  
    	where the denominator of \(Q(u)\) vanishes.  
    	
    The Cauchy integral formula for the derivative yields:  
    \begin{align*}  
    	\sup_{\Vert h \Vert_s = 1} \Vert DX_Q(u) h \Vert_s  
    	&= \Big\Vert \frac{1}{2\pi \mathrm{i}} \int_{|\zeta| = \rho \Vert u \Vert_s}  
    	\frac{X_Q(u + \zeta h)}{\zeta^2} \, \mathrm{d}\zeta \Big\Vert_s \\  
    	&\leq \frac{1}{\rho \Vert u \Vert_s} \sup_{\Vert u' - u \Vert_s = \rho \Vert u \Vert_s} \Vert X_Q(u') \Vert_s.  
    \end{align*}  
    \end{proof}

    \begin{lemma}[Estimate for Lie bracket]\label{fracLie}  
    	For \(Q_1 \in \mathcal{H}^M_{q_1}\) and \(Q_2 \in \mathcal{H}^M_{q_2}\), with \(\gamma'\) defined as in Lemma \ref{Cauchy}, we have  
    	\(\{Q_1, Q_2\} \in \mathcal{H}^M_{q_1+q_2-2}\), and the following estimates hold:  
    	\[
    	|\{Q_1,Q_2\}|_{[r,s,\gamma]} \leq \frac{2+2\rho}{\rho} |Q_1|_{[r+r\rho,s,\gamma']} |Q_2|_{[r+r\rho,s,\gamma']},
    	\]
    	\[
    	|ad_{Q_1}^l Q_2|_{[r,s,\gamma]} \leq \left(\frac{2+2\rho}{\rho} |Q_1|_{[r+r\rho,s,\gamma']} \right)^l |Q_2|_{[r+r\rho,s,\gamma']}.
    	\]
    \end{lemma}
     
     \begin{proof}  
     	We begin by observing the following bound for the vector field \(X_{\{Q_1, Q_2\}}(u)\):  
     	\begin{align*}  
     		\Vert X_{\{Q_1,Q_2\}}(u) \Vert_s &\leq \Vert DX_{Q_1}(u) X_{Q_2} - DX_{Q_2}(u) X_{Q_1} \Vert_s \\  
     		&\leq \Vert DX_{Q_1}(u) X_{Q_2} \Vert_s + \Vert DX_{Q_2}(u) X_{Q_1} \Vert_s.  
     	\end{align*}  
     	
     	Without loss of generality, we estimate the first term using Lemma \ref{Cauchy}:  
     	\begin{align*}  
     		\Vert DX_{Q_1}(u) X_{Q_2} \Vert_s  
     		&\leq \frac{1}{\rho \Vert u \Vert_s} \Vert X_{Q_2}(u) \Vert_s  
     		\sup_{\Vert u' - u \Vert_s = \rho \Vert u \Vert_s} \Vert X_Q(u') \Vert_s \\  
     		&\leq \frac{1}{\rho \Vert u \Vert_s} \Vert X_{Q_2}(u) \Vert_s  
     		\sup_{\Vert u' - u \Vert_s = \rho \Vert u \Vert_s}  
     		\frac{\Vert X_Q(u') \Vert_s}{\Vert u' \Vert}  
     		\sup_{\Vert u' - u \Vert_s = \rho \Vert u \Vert_s} \Vert u' \Vert_s \\  
     		&\leq \frac{1+\rho}{\rho \Vert u \Vert_s} \Vert X_{Q_2}(u) \Vert_s  
     		\sup_{\Vert u' - u \Vert_s = \rho \Vert u \Vert_s}  
     		\frac{\Vert X_Q(u') \Vert_s}{\Vert u' \Vert} \Vert u \Vert.  
     	\end{align*}  
     	
     	Next, taking the supremum over \(u \in \mathfrak{B}_{\gamma,s}^{\mathfrak{h}_Q, M}(r)\), we obtain:  
     	\begin{align*}  
     		\sup_{u \in \mathfrak{B}_{\gamma,s}^{\mathfrak{h}_Q, M}(r)}  
     		\frac{1}{\Vert u \Vert_s} \Vert DX_{Q_1}(u) X_{Q_2} \Vert_s  
     		&\leq \frac{1+\rho}{\rho}  
     		\sup_{\substack{u \in \mathfrak{B}_{\gamma,s}^{\mathfrak{h}_Q, M}(r) \\ \Vert u' - u \Vert_s = \rho \Vert u \Vert_s}}  
     		\frac{1}{\Vert u' \Vert_s} \Vert X_{Q_1}(u') \Vert_s  
     		\sup_{u \in \mathfrak{B}_{\gamma,s}^{\mathfrak{h}_Q, M}(r)}  
     		\frac{1}{\Vert u \Vert_s} \Vert X_{Q_2}(u) \Vert_s \\  
     		&\leq \frac{1+\rho}{\rho} |Q_1|_{[r+r\rho, s, \gamma']} |Q_2|_{[r+r\rho, s, \gamma']}.  
     	\end{align*}  
     	
     	A similar estimate holds for \(\Vert DX_{Q_2}(u) X_{Q_1} \Vert_s\). Combining these bounds yields the desired result.  
     \end{proof}

    \begin{lemma}[Flow lemma]\label{fracflow}  
    	Let \(S\) be a rational normal form defined on \(\mathfrak{B}_{\gamma,s}^{\mathfrak{h}_Q,M}(r)\), satisfying the bound:  
    	\[
    	|S|_{[r+r\rho,s,\gamma']} < \min\left\{\frac{\rho}{1+\rho}, \frac{\gamma - \gamma'}{2d(1+\rho)^2} \right\}.
    	\]    
    	Then, for any \(u \in \mathfrak{B}_{\gamma,s}^{\mathfrak{h}_Q,M}(r)\), there exists a canonical symplectic transformation  
    	\[
    	\Psi_S^t(u): [-1,1] \times \mathfrak{B}_{\gamma,s}^{\mathfrak{h}_Q,M}(r) \to \mathfrak{B}_{\gamma',s}^{\mathfrak{h}_Q,M}(r + r\rho)
    	\]  
    	satisfying the following properties:
    	\begin{enumerate}  
    		\item \textbf{Cauchy problem:}  
    		\[
    		\begin{cases}  
    			\partial_t \Psi_S^t(u) = X_S(\Psi_S^t(u)), \\  
    			\Psi_S^0(u) = u;
    		\end{cases}
    		\]
    		\item \textbf{Locally invertible:}  
    		\[
    		\Psi_S^{-t} \circ \Psi_S^t(u) = u, \quad \forall \Psi_S^t(u) \in \mathfrak{B}_{\gamma,s}^{\mathfrak{h}_Q,M}(r);
    		\]  
    		\item \textbf{Nearly identity:}  
    		\[
    		\Vert \Psi_S^t(u) - u \Vert_s \leq \rho \Vert u \Vert_s.
    		\]
    	\end{enumerate}  
    \end{lemma}

    \begin{proof}
    	We consider the maximal solution \(y = y(t) \in C^1([0, T), W_s)\) of the initial value problem:
    	\[
    	\begin{cases}
    		\frac{\d}{\d t} y(t) = X_S(y(t)), \\
    		y(0) = u, \quad u \in \mathfrak{B}_{\gamma, s}^{\mathfrak{h}_Q, M}(r).
    	\end{cases}
    	\]
    	Define the following two subsets:
    	\[
    	E_1 := [0, T) \cap [0, 1],
    	\]
    	\[
    	E_2 := \left\{ t \in (0, 1] \;\middle|\; \forall \tau \in [0, t], \; \|y(\tau)\|_s \leq (1+\rho)\|u\|_s, \; y(\tau) \in \mathfrak{B}_{\gamma', s}^{\mathfrak{h}_Q, M}(r + r\rho) \right\}.
    	\]
    	
    For \(t \in E_2\), we estimate \(\|y(t) - u\|_s\) as follows:
    \[
    \|y(t) - u\|_s \leq \int_0^t \|X_S(y(\tau))\|_s \, \d \tau,
    \]
    \[
    \leq \int_0^t \|y(\tau)\|_s |S|_{[r + \rho r, s, \gamma']} \, \d \tau,
    \]
    \[
    \leq (1+\rho) \|u\|_s |S|_{[r + \rho r, s, \gamma']} \cdot t,
    \]
    \[
    < \rho \|u\|_s.
    \]
    For every \(\mathcal{J} \in \mathcal{N}\) with \(\#\mathcal{J} \leq \mathfrak{h}_k\), we apply Lemma \ref{lip} to obtain:
    \[
    |\omega_{\mathcal{J}}^M(\Psi_S(u)) - \omega_{\mathcal{J}}^M(u)| \leq  
    2d C_K (1+\rho)^2 \|u\|_s^2 |S|_{[r + \rho r, s, \gamma']}
    \]
    Then
    	\begin{align*}
    		|\omega_{\mathcal{J}}^M\Psi_S(u))|&\geq|\omega_{\mathcal{J}}^M(u)|-|\omega_{\mathcal{J}}^M(\Psi_S(u))-\omega_{\mathcal{J}}^M(u)|\\
    		&\geq\gamma\Vert u\Vert_s^2-2d(1+\rho)^2\Vert u\Vert_s^2|S|_{[r+\rho r,s,\gamma']}\\
    		&>\gamma'\Vert u\Vert_s^2.
    	\end{align*}
    	
    	The estimates above imply that the set \(E_2\) is open in \(E_1\). Furthermore, by definition, \(E_2\) is also closed. Since \(E_1\) is connected, we conclude that \(E_2 = E_1\). 
    	
    	Thus, Properties 2 and 3 are verified as a consequence of the first estimate.
    	
    \end{proof}
    \subsection{Integrable Normal Form Lemma}\label{sec:infl}
    We now focus on the Hamiltonian truncated to the low-mode subspace \(W_s^M\):  
    \begin{equation}\label{fracnorm}
    	 H^M := H_0^M + K_2^M + Z_2^M = (H_0 + Z_d)\big|_{W_s^M}.
    \end{equation}

    Specifically, the components of $H^M$ are expanded as follows:
    \begin{align*}
    	H_0^M&=\sum_{|j|\leq M}j^2|u_j|^2,\\
    	K_2^M&=\sum_{|k_1|\leq M,|k_2|\leq M}K_{|k_1-k_2|}|u_{k_1}|^2|u_{k_2}|^2,\\
    	Z_2^M&=(Z_d-\overline{K}_2)\mid_{W_s^M}.
    \end{align*}
    The following proposition establishes the rational normal form of the truncated Hamiltonian.
    
    \begin{proposition}\label{integr}
    	Let \(d > 5\), \(s > s_0\), and suppose the parameters \(r\) and \(\gamma\) satisfy \(rd^3\gamma^{-1} < 1\). Furthermore, let \(C_2 < 1\) be a small constant such that \(48C_2C_K < 1\).  
    	Consider the Hamiltonian \(H^M\) from \eqref{fracnorm}, defined on the ball \(\mathfrak{B}_{\gamma, s}^{2d, M}(2r)\).
    	Then, there exists a symplectic map \(\Psi_d : \mathfrak{B}_{2\gamma, s}^{2d, M}(r) \to \mathfrak{B}_{\gamma, s}^{2d, M}(2r)\) such that:
    	\begin{enumerate}
    		\item \(H^M \circ \Psi_d^{-1} = H_0^M + K_d^M + \Upsilon_d,\)  
    		where \(K_d^M\) satisfies \(\{K_d^M, |u_J|^2\} = 0\) for all \(J \in \mathfrak{J}^M\);  
    		\item \(\Vert \Psi_d(u) - u \Vert_s \leq 4C_2^2 \Vert u \Vert_s^3 \gamma^{-2} d^3,\)  
    		for all \(u \in \mathfrak{B}_{2\gamma, s}^{2d, M}(r);\)
    		\item \(|\Upsilon_d|_{[r_d, s, \gamma_d]} \leq C_2^{2d} r^{2d - 2} \gamma^{-2d + 2} d^{5d}.\)
    	\end{enumerate}
    \end{proposition}
    
    \begin{proof}
    	We prove this proposition by induction. Define the parameters \(r_k\), \(\gamma_k\), and \(\rho_k\) as follows:
    	\[
    	r_k = 2r - \frac{k-2}{d-2}r, \quad \gamma_k = \frac{k-2}{d-2}\gamma + \gamma, \quad 
    	\rho_k = \frac{r_k - r_{k+1}}{r_{k+1}} = \frac{1}{2d - k - 3}.
    	\]
    	For each step \(k\) with \(2 \leq k \leq d\), we construct a transformation 
    	\[
    	\Psi_k : \mathfrak{B}_{\gamma_{k+1}, s}^{2d, M}(r_{k+1}) \to \mathfrak{B}_{\gamma_k, s}^{2d, M}(r_k),
    	\]
    	which satisfies the following conditions:
    	
        \begin{enumerate}
        	\item \(H^M_k := H^M \circ \Psi_k^{-1} = H_0^M + K_k^M + Z_k^M + \Upsilon_k,\)  
        	where \(H^M_k\) is defined on \(\mathfrak{B}_{\gamma_k, s}^{2d, M}(r_k)\);
        	\item \(|Z_k^M|_{[r_k, s, \gamma_k]} \leq C_2^k r^{2k} \gamma^{4 - 2k} d^{2k - 3},\)  
        	and \(Z_k^M \in \mathcal{H}_{2k+2}^M\);
        	\item \(\Vert \Psi_k(u) - u \Vert_s \leq \sum_{l=2}^{k-1} 2C_2^l \Vert u \Vert_s^{2l-1} \gamma^{2 - 2l} d^{2l-1}\)  
        	for all \(u \in \mathfrak{B}_{\gamma_{k+1}, s}^{2d, M}(r_{k+1})\);
        	\item \(|K_k^M - K_2^M|_{[r_k, s, \gamma_k]} \leq 2C_2^2 d r^4\).
        \end{enumerate}

    	For the base case $k=2$, we set $\Psi_2=Id$ and $\Upsilon_2=0.$ 
    	We only need to verify (4). By Lemma \ref{normP}, we have
    	\[
    	\Vert X_{K_2^M}(u) \Vert_s \leq C_K \Vert u \Vert_s^3,  
    	\]
    	\[
    	|K_2^M|_{[r_2, s, \gamma_2]} = \sup_{u \in \mathfrak{B}_{\gamma_2, s}^{2d, M}} \frac{\Vert X_{K_2^M} \Vert_s}{\Vert u \Vert_s} \leq C_K r_2^2.
    	\]

        From the estimate for \(Z_k\) in the proof of Proposition \ref{resonantlize}, it follows that \(Z_2^M\) is the truncated \(Z_d\). Consequently,  
        \[
        |Z_2^M|_{[r_2,s,\gamma]} \leq 2C_Z d r_2^2.
        \]

		Now, assume the inductive hypothesis holds for some \(k \geq 2\). We will prove it for \(k+1\) by solving the following homological equation:
		\[
		\{K_2^M, S_k^M\} + Z_k^M = \Delta K_k^M.
		\]
		
    	The solution is given by  
    	\[
    	S_k^M = \sum_{\mathcal{J} \in \mathcal{R} \cap \mathfrak{J}^M} S_{\mathcal{J}}(u) u_{\mathcal{J}}  
    	= \sum_{\mathcal{J} \in \mathcal{R} \cap \mathfrak{J}^M} \frac{\i Z_{k,\mathcal{J}}^M(u)}{\omega_{\mathcal{J}}^M(u)} u_{\mathcal{J}},
    	\]

    	\begin{align*}
    		\Vert X_{S_k^M}\Vert_s&=\sum_{J\in\mathcal{Z}}\sum_{\mathcal{J}\in\mathcal{R}\cap\mathfrak{J}^M}(|\frac{1}{\omega_{\mathcal{J}}^M(u)}\frac{\partial Z_{k,\mathcal{J}}^M(u)u_{\mathcal{J}}}{\partial u_{J}}|+C_K|\frac{ Z_{k,\mathcal{J}}^M(u)u_{\mathcal{J}}u_{\overline{J}}}{(\omega_{\mathcal{J}}^M(u))^2}|)\e^{sf(|J|)}\\
    		&\leq\sum_{J\in\mathcal{Z}}\sum_{\mathcal{J}\in\mathcal{R}\cap\mathfrak{J}^M}(\frac{1}{\gamma_k\Vert u\Vert_s^2}|\frac{\partial Z_{k,\mathcal{J}}^M(u)u_{\mathcal{J}}}{\partial u_{J}}|+\frac{C_K}{\gamma_k^2\Vert u\Vert_s^2}| Z_{k,\mathcal{J}}^M(u) u_{\mathcal{J}})\e^{sf(|J|)}\\
    		&\leq\frac{2C_K}{\gamma_k^2\Vert u\Vert_s^2}\Vert X_{Z_k^M}\Vert_s.
    	\end{align*}
    	By our choice of parameters $r, d, \gamma$, it follows that:
    	$$|S_k^M|_{[r_k,s,\gamma_k]}\leq \frac{2C_K}{\gamma_{k}^2}C_2^kr^{2k-2}\gamma^{4-2k}d^{2k-3}\leq\frac{1}{2d}\leq\min\{\frac{\rho_k}{1+\rho_k},\frac{\gamma_{k+1}-\gamma_{k}}{2d(1+\rho_k)^2}\}.$$
    	
    	We define \(\psi_k\) as the time-1 map of the flow generated by \(S_k^M\). By Lemma \ref{fracflow}, the difference of \(\psi_k^{\pm}\) from the identity can be estimated as follows:
    	
    	\begin{align*}
    		\Vert \psi_k(u)-u\Vert_s&\leq\Vert\int_{0}^{1}X_{S_k^M}(u(\tau))\d\tau\Vert_s\\
    	&\leq\Vert X_{S_k^M}(u(\tau))\Vert_s\\
    	&\leq2C_KC_2^k\Vert u\Vert_s^{2k-1}\gamma^{2-2k}d^{2k-1}\\
    	&\leq\rho_k\Vert u\Vert_s.
    	\end{align*}

    	This estimate holds for all \(u \in \mathfrak{B}_{\gamma_{k+1},s}^{2d,M}(r_{k+1})\).	Moreover, Lemma \ref{Cauchy} ensures that \(\psi_k^{\pm}\) maps  
    	\(\mathfrak{B}_{\gamma_{k+1},s}^{2d,M}(r_{k+1})\) into \(\mathfrak{B}_{\gamma_k,s}^{2d,M}(r_k)\).
    	
    	For the \((k+1)\)-th step, we define the transformation as the composition  
    	\[
    	\Psi_{k+1} = \Psi_k \circ \psi_k.
    	\]
    	Then, using the telescope argument, we can obtain the nearly-identity property of \(\Psi_k\):
    	
    	\begin{align*}
    		\Vert\Psi_{k+1}(u)-u\Vert_s&\leq\Vert\Psi_k\circ\psi_k(u)-\psi_k(u)\Vert_s+\Vert\psi_k(u)-u\Vert_s\\
    		&\leq\sum_{l=2}^{k-1}2C_2^l\Vert u\Vert_s^{2l-1}\gamma^{2-2l}d^{2l-1}+2C_2^k\Vert u\Vert_s^{2k-1}\gamma^{2-2k}d^{2k-1}\\
    		&\leq\sum_{l=2}^{k}2C_2^l\Vert u\Vert_s^{2l-1}\gamma^{2-2l}d^{2l-1}\\
    		&\leq \frac{2C_2^2\Vert u\Vert_s^3\gamma^{-2}d^3}{1-C_2\Vert u\Vert_s^2\gamma^{-2}d^2}\leq4C_2^2\Vert u\Vert_s^3\gamma^{-2}d^3.
    	\end{align*}
    
        The transformed Hamiltonian \(H_{k+1}^M = H_k^M \circ \psi_k^{-1}\),  
        which is defined on \(\mathfrak{B}_{\gamma_{k+1},s}^{2d,M}(r_{k+1})\),  
        can be expressed via Taylor's formula with an integral remainder:
        
        \begin{align*}
        	H_k^M \circ \psi_k^{-1} 
        	&= H_0^M + K_2^M + (K_k^M - K_2^M) + Z_k^M \\
        	&\quad + \{K_2^M, S_k^M\} + \sum_{l=2}^{n_1} \frac{ad_{S_k^M}^l}{l!} K_2^M + R^M_{K_2,k} \\
        	&\quad + \sum_{l=1}^{n_2} \frac{ad_{S_k^M}^l}{l!} (K_k^M - K_2^M) + R^M_{K_k,k} \\
        	&\quad + \sum_{l=1}^{n_3} \frac{ad_{S_k^M}^l}{l!} Z_k^M + R^M_{Z_k,k} + \Upsilon_k \circ \psi_k^{-1},
        \end{align*}
        
    	where
    	\begin{align*}
    		R_{K_2,k}^M&=\int_{0}^{1}\frac{(1-\tau)^{n_1}}{n_1!}ad_{S_k}^{n_1+1}(K_2^M)\circ\psi_k^{-\tau}(u)\d\tau,\\
    		R_{K_k,k}^M&=\int_{0}^{1}\frac{(1-\tau)^{n_2}}{n_2!}ad_{S_k}^{n_2+1}(K_k^M)\circ\psi_k^{-\tau}(u)\d\tau,\\
    		R_{Z_k,k}^M&=\int_{0}^{1}\frac{(1-\tau)^{n_3}}{n_3!}ad_{S_k}^{n_3+1}(Z_k^M)\circ\psi_k^{-\tau}(u)\d\tau.\\
    	\end{align*}
    	The integers \(n_1, n_2, n_3\) are chosen sufficiently large such that  
    	the degree of the polynomial in the integrands of the remainder terms  
    	is the smallest integer greater than \(d\).

    	We group the terms as follows:
    	\[
    	Z_{k+1}^M = \sum_{l=2}^{n_1} \frac{ad_{S_k^M}^l}{l!} K_2^M 
    	+ \sum_{l=1}^{n_2} \frac{ad_{S_k^M}^l}{l!} (K_k^M - K_2^M) 
    	+ \sum_{l=1}^{n_3} \frac{ad_{S_k^M}^l}{l!} Z_k^M,
    	\]
    	\[
    	K_{k+1}^M = K_k^M + \Delta K_k^M, \qquad 
    	R_{k+1}^M = R_{K_2,k}^M + R_{K_k,k}^M + R_{Z_k,k}^M + \Upsilon_k \circ \psi_k^{-1}.
    	\]
    	
    Setting \(\delta_k = \frac{r_k}{r_{k+1}} = \frac{1 + \rho_k}{\rho_k} = 2d - k - 2,\)  
    we can now estimate the components of \(Z_{k+1}\) using Lemma \ref{fracLie},  
    the choice of constants, and our inductive assumptions:
    \begin{align*}
    	\Bigg|\sum_{l=2}^{n_1} \frac{ad_{S_k^M}^l}{l!} K_2^M \Bigg|_{[r_{k+1}, s, \gamma_{k+1}]}
    	&\leq \Bigg|\sum_{l=2}^{n_1} \frac{ad_{S_k^M}^{l-1}}{l!} \{K_2^M, S_k^M\} \Bigg|_{[r_{k+1}, s, \gamma_{k+1}]} \\
    	&\leq \sum_{l=1}^{n_1-1} \frac{1}{(l+1)!} \Big( 2 \delta_k |S_k^M|_{[r_k, s, \gamma_k]} \Big)^l 
    	\Big| \{K_2^M, S_k^M\} \Big|_{[r_k, s, \gamma_k]} \\
    	&\leq \Big(\exp\big(2 \delta_k |S_k^M|_{[r_k, s, \gamma_k]}\big)\Big) 
    	2 \delta_k |S_k^M|_{[r_k, s, \gamma_k]} |Z_k^M|_{[r_k, s, \gamma_k]} \\
    	&\leq 8d \frac{2C_K}{\gamma_k^2} C_2^{2k} r^{4k-2} \gamma^{8 - 4k} d^{4k-6} \\
    	&\leq \frac{1}{3} C_2^{k+1} \gamma^{2-2k} r^{2k+2} d^{2k+1};
    \end{align*}
    
    \begin{align*}
    	|\sum_{l=1}^{n_2}\frac{ad_{S_k^M}^l}{l!}(K_k^M-K_2^M)|_{[r_{k+1},s,\gamma_{k+1}]}&\leq\sum_{l=1}^{n_1}\frac{1}{l!}(2\delta_k|S_k^M|_{[r_{k},s,\gamma_{k}]})^{l}|(K_k^M-K_2^M)|_{[r_k,s,\gamma_k]}\\
    	&\leq\sum_{l=1}^{n_1}\frac{1}{l!}(2\delta_k|S_k^M|_{[r_{k},s,\gamma_{k}]})^{l-1}2\delta_k|S_k^M|_{[r_{k},s,\gamma_{k}]}|(K_k^M-K_2^M)|_{[r_k,s,\gamma_k]}\\
    	&\leq(\e^{2\delta_k|S_k^M|_{[r_{k},s,\gamma_{k}]}})2\delta_k|S_k^M|_{[r_{k},s,\gamma_{k}]}|(K_k^M-K_2^M)|_{[r_k,s,\gamma_k]}\\
    	&\leq8d\frac{2C_K}{\gamma_{k}^2}C_2^{k}r^{2k-2}\gamma^{8-4k}d^{4k-6}2C_Kdr^4\\
    	&\leq\frac{1}{3}C_2^{k+1}\gamma^{2-2k} r^{2k+2}d^{2k+1};
    \end{align*}
    
       \begin{align*}
       	\Bigg|\sum_{l=1}^{n_2} \frac{ad_{S_k^M}^l}{l!} Z_k^M \Bigg|_{[r_{k+1}, s, \gamma_{k+1}]}
       	&\leq \sum_{l=1}^{n_1} \frac{1}{l!} \Big( 2 \delta_k |S_k^M|_{[r_k, s, \gamma_k]} \Big)^l 
       	|Z_k^M|_{[r_k, s, \gamma_k]} \\
       	&\leq \sum_{l=1}^{n_1} \frac{1}{l!} \Big( 2 \delta_k |S_k^M|_{[r_k, s, \gamma_k]} \Big)^{l-1} 
       	2 \delta_k |S_k^M|_{[r_k, s, \gamma_k]} |Z_k^M|_{[r_k, s, \gamma_k]} \\
       	&\leq \Big( \exp\big( 2 \delta_k |S_k^M|_{[r_k, s, \gamma_k]} \big) \Big) 
       	2 \delta_k |S_k^M|_{[r_k, s, \gamma_k]} |Z_k^M|_{[r_k, s, \gamma_k]} \\
       	&\leq 8d \frac{2C_K}{\gamma_k^2} C_2^{2k} r^{4k-2} \gamma^{8-4k} d^{4k-6} \\
       	&\leq \frac{1}{3} C_2^{k+1} \gamma^{2-2k} r^{2k+2} d^{2k+1}.
       \end{align*}

    	 Summing the above three bounds gives the desired estimate for \(Z_{k+1}^M.\)

    	 Next, we estimate \(K_{k+1}^M - K_2^M:\)
    	 \begin{align*}
    	 	\Big| K_{k+1}^M - K_2^M \Big|_{[r_{k+1}, s, \gamma_{k+1}]}
    	 	&\leq \sum_{l=2}^{k} \Big|\Delta K_l \Big|_{[r_l, s, \gamma_l]} \\
    	 	&\leq \sum_{l=2}^{k} \Big| Z_k^M \Big|_{[r_l, s, \gamma_l]} \\
    	 	&\leq \sum_{l=2}^{k} C_2^l r^{2l} \gamma^{4 - 2l} d^{2l-3} \\
    	 	&\leq \frac{C_2^2 r^4 d}{1 - C_2 r^2 \gamma^{-2} d^2} 
    	 	\leq 2C_2^2 d r^4.
    	 \end{align*}

    	To estimate the remainder \(R_{K_2,k}^M\), we choose \(n_1 = \Big[\frac{d-k}{k-1}\Big] + 1\).  
    	This is the smallest integer \(n_1\) such that \(n_1(2k-2) + 2k \geq 2d.\)  
    	Under this choice, we derive the following bound:
    	\begin{align*}
    		\Big|R_{K_2,k}^M\Big|_{[r_{k+1},s,\gamma_{k+1}]} 
    		&\leq \Big| ad_{S_k^M}^{n_1} \{ S_k^M, K_2^M \} \Big|_{[r_{k+1},s,\gamma_{k+1}]} \\
    		&\leq \Big|S_k^M\Big|_{[r_k, s, \gamma_{k}]}^{n_1} \Big|Z_k^M\Big|_{[r_k, s, \gamma_{k}]} \\
    		&\leq \Bigg( \frac{2C_K}{\gamma_{k}^2} C_2^k r^{2k-2} \gamma^{4-2k} d^{2k-1} \Bigg)^{n_1} 
    		C_2^k r^{2k} \gamma^{4-2k} d^{2k-3} \\
    		&\leq (2C_K)^{\frac{d}{2}} C_2^{2d} r^{2d} \gamma^{-2d} d^{5d} \\
    		&\leq \frac{1}{3} C_2^{2d} r^{2d-2} \gamma^{2d-2} d^{5d}.
    	\end{align*}
    	
    	For \(R_{K_k,k}^M\), we take \(n_2 = \Big[\frac{d-1}{k-1}\Big],\)  
    	which is the smallest integer \(n_2\) satisfying \((n_2+1)(2k-2)+2 \geq 2d.\)  
    	Under this setting, we have:
    	\begin{align*}
    		\Big|R_{K_k,k}^M\Big|_{[r_{k+1},s,\gamma_{k+1}]} 
    		&\leq \Big| ad_{S_k^M}^{n_2+1} K_k^M \Big|_{[r_{k+1},s,\gamma_{k+1}]} \\
    		&\leq \Big|S_k^M\Big|_{[r_k, s, \gamma_{k}]}^{n_2+1} \Big|K_k^M\Big|_{[r_k, s, \gamma_{k}]} \\
    		&\leq \Bigg( \frac{2C_K}{\gamma_{k}^2} C_2^k r^{2k-2} \gamma^{4-2k} d^{2k-1} \Bigg)^{n_2+1} 
    		\cdot 2 C_K d r^4 \\
    		&\leq (2C_K)^d C_2^{3d} r^{2d-2} \gamma^{2-2d} d^{3d} \\
    		&\leq \frac{1}{3} C_2^{2d} r^{2d-2} \gamma^{2-2d} d^{5d}.
    	\end{align*}

     	For \(R_{Z_k,k}^M\), we take \(n_3 = \Big[\frac{d-k}{k-1}\Big],\)  
     	which is the smallest integer \(n_3\) such that \((n_3+1)(2k-2)+2k \geq 2d.\)  
     	This leads to the following estimate:
     	\begin{align*}
     		\Big|R_{Z_k,k}^M\Big|_{[r_{k+1},s,\gamma_{k+1}]} 
     		&\leq \Big| ad_{S_k^M}^{n_3+1} Z_k^M \Big|_{[r_{k+1},s,\gamma_{k+1}]} \\
     		&\leq \Big|S_k^M\Big|_{[r_k, s, \gamma_{k}]}^{n_3+1} \Big|Z_k^M\Big|_{[r_k, s, \gamma_{k}]} \\
     		&\leq \Bigg( \frac{2C_K}{\gamma_{k}^2} C_2^k r^{2k-2} \gamma^{4-2k} d^{2k-1} \Bigg)^{n_3+1} 
     		C^k r^{2k} \gamma^{4-2k} d^{2k-3} \\
     		&\leq (2C_K)^{\frac{d}{2}} C_2^{2d} r^{2d} \gamma^{-2d} d^{5d} \\
     		&\leq \frac{1}{3} C_2^{2d} r^{2d-2} \gamma^{2d-2} d^{5d}.
     	\end{align*}

    	Finally, we combine the estimates for the remainder terms to bound \(\Upsilon_{k+1}\) as follows:
    	\begin{align*}
    		\Big|\Upsilon_{k+1}\Big|_{[r_{k+1},s,\gamma_{k+1}]}
    		&\leq \Big|R_{K_2,k}^M\Big|_{[r_{k+1},s,\gamma_{k+1}]} 
    		+ \Big|R_{K_k,k}^M\Big|_{[r_{k+1},s,\gamma_{k+1}]} 
    		+ \Big|R_{Z_k,k}^M\Big|_{[r_{k+1},s,\gamma_{k+1}]} 
    		+ \Big|\Upsilon_k \circ \psi_k^{-1}\Big|_{[r_{k+1},s,\gamma_{k+1}]} \\
    		&\leq C_2^{2d} r^{2d-2} \gamma^{2d-2} d^{5d} 
    		+ \sum_{l=0}^{\infty} \frac{1}{l!} 
    		\big( 2 \delta_k \big|S_k\big|_{[r_k,s,\gamma_{k}]} \big)^l 
    		\big|\Upsilon_k\big|_{[r_k,s,\gamma_{k}]} \\
    		&\leq C_2^{2d} r^{2d-2} \gamma^{2d-2} d^{5d} + 2 \big|\Upsilon_k\big|_{[r_k,s,\gamma_{k}]}.
    	\end{align*}
    	
Rearranging this inequality, we obtain:
\begin{align*}
	\frac{1}{2^{k+1}} \Big|\Upsilon_{k+1}\Big|_{[r_{k+1},s,\gamma_{k+1}]}
	&\leq \frac{1}{2^{k+1}} C_2^{2d} r^{2d-2} \gamma^{2d-2} d^{5d} 
	+ \frac{1}{2^k} \Big|\Upsilon_k\Big|_{[r_k, s, \gamma_k]},\\
     |\Upsilon_d|_{[r_d,s,\gamma_{d}]}\frac{1}{2^d}
     &\leq\sum_{l=3}^{d}\frac{C_2^{2d}r^{2d-2}\gamma^{2d-2}d^{5d}}{2^l},\\
     |\Upsilon_d|_{[r_d,s,\gamma_{d}]}&\leq C_2^{2d}r^{2d-2}\gamma^{2d-2}d^{5d}.
     \end{align*}
    \end{proof}

    \section{Stability Estimate}\label{sec:sta}
    Let \(u(t): [0,T^*) \to W^s\) be the maximal solution to the original system with initial data 
    \(u(0) \in B_s\Big(\frac{r}{2}\Big) \cap \Big(\Pi^{M}\Big)^{-1} \mathfrak{D}_{3\gamma}^{2d,M}\).  
    We employ a standard bootstrap argument and define  
    \[
    T := \sup \Big\{ t \in [0,T^*] \mid \forall \tau \in [0,t], \, 
    \sum_{J \in \mathcal{Z}} \e^{sf(|J|)} \Big| \big|u_J(\tau)\big|^2 - \big|u_J(0)\big|^2 \Big|^{\frac{1}{2}} 
    \leq \Vert u(0)\Vert_s^{\frac{3}{2}} \Big\}.
    \]
    
    Our goal is to show that \(T > T_r\), where \(T_r\) will be introduced in Section \ref{sec:len}.  
    The value \(T_r\) depends on the space \(W_s\), namely the regularity of the solution and the nonlinear term \(K\).  
    We will argue by contradiction: suppose \(T \leq T_r\). This assumption implies that the bootstrap condition must be violated at time \(T_r\). Therefore, the following inequality holds at time \(T\):
    \[
    \sum_{J \in \mathcal{Z}} \e^{sf(|J|)} \Big| \big|u_J(T)\big|^2 - \big|u_J(0)\big|^2 \Big|^{\frac{1}{2}} 
    \geq \Vert u(0)\Vert_s^{\frac{3}{2}}.
    \]  
    The remainder of this section will prove the opposite of this inequality to arrive at a contradiction.

    For any \(t \in [0,T]\), the definition of \(T\) leads to the following immediate estimate:  
    
    \begin{align*}
    	\Vert u(t)\Vert_s&\leq\Vert u(0)\Vert_s+\sum_{J\in\mathcal{Z}}\e^{sf(|J|)}||u(t)|-|u(0)||\leq\Vert u(0)\Vert_s+\Vert u(0)\Vert_s^{\frac{3}{2}},\\
    	\sum_{J\in\mathcal{Z}}\e^{2sf(|J|)}||u_J(t)|^2-|u_J(0)|^2|&\leq(\sum_{J\in\mathcal{Z}}\e^{sf(|J|)}||u_J(t)|-|u_J(0)||^{\frac{1}{2}})^2\leq\Vert u(0)\Vert_s^{3}.
    \end{align*}
    Using these bounds, we now demonstrate that the solution remains within a non-resonant region.  
    Define \(\omega_{\mathcal{J}}(u) := \omega_{\mathcal{J}}(\Pi^{M} u)\), where \(\mathcal{J} \in \mathfrak{J}^{d,M}\). We establish the following inequality:

    \begin{align*}
    	|\omega_{\mathcal{J}}(u(t))| &\geq |\omega_{\mathcal{J}}(u(0))| - \Big| \omega_{\mathcal{J}}(u(t)) - \omega_{\mathcal{J}}(u(0)) \Big|
    	\\
    	&\geq3\gamma \Vert u(0) \Vert_s^2 - \sum_{J \in \mathcal{J}} \Big| |u_J(t)|^2 - |u_J(0)|^2 \Big|.
    	\\
    	&\geq \big( 3\gamma - \Vert u(0) \Vert_s \big) \Vert u(0) \Vert_s^2\\
    	&\geq2\gamma\Vert u(t)\Vert_s^2.
    \end{align*}
    We now apply the transformations derived in the previous sections.  
    First, applying the map from Proposition \ref{resonantlize}, we define the transformed variable \(v(t) := \Phi_d(u(t))\).  
    This new variable \(v\) evolves according to the Hamiltonian system generated by  
    \[
    H \circ \Phi_d^{-1} = H_0 + Z_d + R_d.
    \]  
    Thus, \(v(t)\) is the solution to the following Cauchy problem:
    
    \[
    \i \partial_t v = \nabla H_0(v) + \nabla Z_d(v) + \nabla R_d(v), \quad v(0) = \Phi_d(u(0)).
    \]
     Next, we decompose \(v\) into its low- and high-mode components as  
     \[
     v = v^{<M} + v^{>M},
     \]  
     which induces the following decomposition of the Hamiltonian:
      \[
      H \circ \Phi_d^{-1}(v) = H^M(v^{<M}) + H^{>M}(v^{<M}, v^{>M}) + R_d(v).
      \]
      By applying Proposition \ref{integr}, we transform the low-mode Hamiltonian \(H^M\) using the map \(\Psi_d\).  
      This results in the rational normal form on \(\mathfrak{B}^{2d,M}_{\gamma,s}(2r)\), expressed as:
      \[
      H^M \circ \Psi_d^{-1} = H_0^M + K_d^M + \Upsilon_d^M.
      \]
      
      We define the variable \(w(t) := \Psi_d(v^M)\) and fix the high-mode component \(w^{>M}\).  
      Under this setting, the evolution of \(w(t)\) satisfies the equation:
      \[
      \i \partial_t w(t) = \nabla (H^M \circ \Psi_d^{-1})(w) + D\Psi_d(v^{<M}) \cdot (\Pi^{<M} X_{H^{>M}}(v)) + D\Psi_d(v^{<M}) \cdot (\Pi^{<M} X_{R_d}(v)),
      \]
      which we rewrite as:
      \[
      \i \partial_t w(t) := \nabla (H^M \circ \Psi_d^{-1})(w) + \mathcal{W}(t).
      \]
      
      Furthermore, we verify that the initial data \(w(0)\) lies in the appropriate non-resonant domain,  
      provided that \(u \in B_s\big(\frac{r}{2}\big) \cap (\Pi^M)^{-1} \mathfrak{D}_{3\gamma}^{2d,M}.\)
      In fact, the following inequality holds for the initial data \(w(0)\):
      \begin{align*}
   	|\omega^M_{\mathcal{J}}(w(0))|&\geq|\omega^M_{\mathcal{J}}(u(0))|-|\omega^M_{\mathcal{J}}(u(0))-\omega^M_{\mathcal{J}}(v(0))|-|\omega^M_{\mathcal{J}}(v(0))-\omega^M_{\mathcal{J}}(w(0))|\\
   	&\geq3\gamma\Vert u(0)\Vert_s^2-2dC_K(4C_KC_2^2\Vert u(0)\Vert_s^3\gamma^{-2}d^3+16C_K\Vert v(0)\Vert_s^3)\\
   	&\geq2\gamma\Vert w(0)\Vert_s^2.
   \end{align*} 
   Thus, we conclude that the initial data \(w(0)\) belongs to the non-resonant domain  
   \(\mathfrak{B}_{2\gamma,s}^{2d,M}(r).\)
  
    To complete the bootstrap argument, it is necessary to show that neither the high-mode nor the low-mode components grow sufficiently to violate the bootstrap condition by time \(T_r\).  
    In particular, we aim to prove the following bounds:
    \[
    \sum_{|J| > M} \e^{sf(|J|)} \Big||v_J(t)|^2 - |v_J(0)|^2 \Big|^{\frac{1}{2}} \leq \frac{1}{4} \Vert u(0) \Vert_s^{\frac{3}{2}},
    \]
    \[
    \sum_{|J| \leq M} \e^{sf(|J|)} \Big||w_J(t)|^2 - |w_J(0)|^2 \Big|^{\frac{1}{2}} \leq \frac{1}{4} \Vert u(0) \Vert_s^{\frac{3}{2}}.
    \]
    
    We begin with the high modes (\(|J| > M\)).  
    The time evolution of \(|v_J|^2\) is described by the following equation:
    \[
    \frac{\d |v_J|^2}{\d t} = \{|v_{J}|^2, Z_d\} + \{|v_{J}|^2, R_d\}.
    \]
    
   By applying Lemma \ref{cut}, the contribution from \(Z_d\) can be bounded as follows:
   \begin{align*}
   	\sum_{|J| > M} \e^{sf(|J|)} \Big| \{ |v_J|^2, Z_d \} \Big|^{\frac{1}{2}}
   	&\leq 2 \sum_{|J| > M} \e^{sf(|J|)} |v_J|^{\frac{1}{2}} \cdot |(X_{Z_d}(v))_J|^{\frac{1}{2}}\\
   	&\leq 2 \Vert v \Vert_s^{\frac{1}{2}} \Vert X_{Z_d} \Vert_s^{\frac{1}{2}}.
   \end{align*}
   Thus, we obtain the following estimate for the high-mode contribution of \(Z_d\):
   \[
   \sum_{|J| > M} \e^{sf(|J|)} \Big| \{ |v_J|^2, Z_d \} \Big|^{\frac{1}{2}} 
   \leq C_Z \e^{\big( \frac{s_0 - s}{2} \big) f\big( \sqrt{\frac{N}{d-2}} \big)}.
   \]
   
     We define the following parameter:
     \[
     \mathfrak{r} := \frac{C_m r^{\frac{2}{5}}}{\gamma}, \quad C_m := 4 \max \{ C_KC_2^2, 1, C_K \}.
     \]
     Using the estimate for \(R_d\) from Proposition \ref{resonantlize}, we evaluate the contribution of \(R_d\) to the high modes:
     \begin{align*}
     	\sum_{|J| > M} \e^{sf(|J|)} \Big| \{|v_J|^2, R_d\} \Big|^{\frac{1}{2}} 
     	&\leq 2 \sum_{|J| > M} \e^{sf(|J|)} |v_J|^{\frac{1}{2}} |(X_{R_d}(v))_J|^{\frac{1}{2}}\\
     	&\leq 2 \Vert v \Vert_s^{\frac{1}{2}} \Vert X_{R_d} \Vert_s^{\frac{1}{2}}\leq 6C_K^d \Vert u(0) \Vert_s^{d-1} d^{\frac{5d}{2}}\\
     	&\leq (C_K^d r^{\frac{2d}{5}} d^d)^{\frac{5}{2}}
     	\leq \mathfrak{r}^d d^d \Vert u(0) \Vert_s^{\frac{3}{2}}.
     \end{align*}
     
      By integrating from \(0\) to \(t \leq T_r\), and using the choice of \(T_r\) from Section \ref{sec:len},  
      we derive the following bound:
      \begin{align*}
      	\sum_{|J| > M} \e^{sf(|J|)} \Big| |v_J(t)|^2 - |v_J(0)|^2 \Big|^{\frac{1}{2}} 
      	&\leq T \sup_{t \in [0,T]} \sum_{|J| > M} \e^{sf(|J|)} \Big| \frac{\d |v_J(t)|^2}{\d t} \Big|^{\frac{1}{2}}\\
      	&\leq T_r \sum_{|J| > M} \e^{sf(|J|)} \Big( \Big| \{ |v_J|^2, Z_d \} \Big|^{\frac{1}{2}} + \Big| \{ |v_J|^2, R_d \} \Big|^{\frac{1}{2}} \Big)\\
      	&\leq T_r \Big( C_Z \e^{\big( \frac{s_0 - s}{2} \big) f\big( \sqrt{\frac{N}{d-2}} \big)} + \mathfrak{r}^d d^d \Vert u(0) \Vert_s^{\frac{3}{2}} \Big)\\
      	&\leq \Vert u(0) \Vert_s^{\frac{3}{2}}.
      \end{align*}

        We now turn our attention to the low modes (\(|J| \leq M\)).  
        The time evolution of \(|w_J|^2\) is governed by the following equation:
        \[
        \frac{\d |w_J(t)|^2}{\d t} = \{ |w_J(t)|^2, H^M \circ \Psi_d^{-1} \} + \Im (w_{\overline{J}} \mathcal{W}_J(t)),
        \]
        where \(\Im(u)\) denotes the imaginary part of \(u\).
        
         From Proposition \ref{integr}, we have \(\{ |w_J|^2, H_0^M + K_d^M \} = 0\).  
         This allows us to simplify the first term as follows:
         \[
         |\{ |w_J(t)|^2, H^M \circ \Psi_d^{-1} \}| = |\{ |w_J(t)|^2, \Upsilon_d \}|
         \leq 2 |w_J(t)(X_{\Upsilon_d})_J|.
         \]
         For the low modes, we can estimate the first term as:
         \begin{align*}
         	 \sum_{|J| \leq M} \e^{sf(|J|)} |\{ |w_J(t)|^2, H^M \circ \Psi_d^{-1} \}|^{\frac{1}{2}}
         	&\leq 2 \sum_{|J| \leq M} \e^{sf(|J|)} |w_J(t)(X_{\Upsilon_d})_J|^{\frac{1}{2}}\\
         	&\leq 2 \Vert w \Vert_s^{\frac{1}{2}} \Vert X_{\Upsilon_d} \Vert_s^{\frac{1}{2}}
         	\leq (2C_K)^d \Vert w \Vert_s^{d-1} \gamma^{1-d} d^{\frac{5}{2}d},\\
         	&\leq (C_K r^{\frac{2}{5}} \gamma^{-1} d^d)^{\frac{5}{2}} 
         	\leq \mathfrak{r}^d d^d.
         \end{align*}
         
         By taking \(Q = \Vert w(t) \Vert_{L^2}^2\) in Lemma \ref{Cauchy}, we obtain:
         \[
         \Vert \mathcal{W}(t) \Vert_s \leq 2 \Big( \Vert \Pi^{<M} X_{H^{>M}}(v) \Vert_s + \Vert \Pi^{<M} X_{R_d}(v) \Vert_s \Big).
         \]
         This result allows us to estimate the contribution of the perturbation to the low modes:
         \begin{align*}
         	\sum_{|J| \leq M} \e^{sf(|J|)} |\Im(w_{\overline{J}} \mathcal{W}_J(t))|^{\frac{1}{2}} 
         	&\leq 2 \sum_{|J| \leq M} \e^{sf(|J|)} \Big( |w_{\overline{J}} (\Pi^{<M} X_{H^{>M}}(v))_J|^{\frac{1}{2}}
         	+ |w_{\overline{J}} (\Pi^{<M} X_{R_d})_J|^{\frac{1}{2}} \Big)\\
         	&\leq 2 \Vert w \Vert_s^{\frac{1}{2}} \Big( \Vert \Pi^{<M} X_{H^{>M}}(v) \Vert_s^{\frac{1}{2}} + \Vert \Pi^{<M} X_{R_d} \Vert_s^{\frac{1}{2}} \Big)\\
         	&\leq C_Z \e^{\big(\frac{s_0 - s}{2}\big) f\big(\sqrt{\frac{M}{d-2}}\big)} + \mathfrak{r}^d d^d.
         \end{align*}
         By the choice of \(T_r\) in Section \ref{sec:len}, we estimate the total variation of the low modes as:
         \begin{align*}
         	\sum_{|J| \leq M} \e^{sf(|J|)} \Big| |w_J(t)|^2 - |w_J(0)|^2 \Big|^{\frac{1}{2}} 
         	&\leq T \sup_{t \in [0, T]} \sum_{|J| \leq M} \e^{sf(|J|)} \Big| \frac{\d |w_J(t)|^2}{\d t} \Big|^{\frac{1}{2}}\\
         	&\leq T_r \Big( \mathfrak{r}^d d^d + C_Z \e^{\big( \frac{s_0 - s}{2} \big) f\big( \sqrt{\frac{M}{d-2}} \big)} + \mathfrak{r}^d d^d \Big)\\
         	& \leq \Vert u(0) \Vert_s^{\frac{3}{2}}.
         \end{align*}
     
         Combining all the preceding estimates, we obtain the following bound for the total variation of the norm of the original solution \(u(t)\):
         \[
         \sum_{J \in \mathcal{Z}} \e^{sf(|J|)} \Big| |u_J(t)|^2 - |u_J(0)|^2 \Big|^{\frac{1}{2}}
         \leq \sum_{|J| \leq M} \e^{sf(|J|)} \Big| |u_J(t)|^2 - |u_J(0)|^2 \Big|^{\frac{1}{2}}
         + \sum_{|J| > M} \e^{sf(|J|)} \Big| |u_J(t)|^2 - |u_J(0)|^2 \Big|^{\frac{1}{2}}
         \]
         \[
         \leq \sum_{|J| \leq M} \e^{sf(|J|)} \Big( \Big| |u_J(t)|^2 - |v_J(t)|^2 \Big|^{\frac{1}{2}} 
         + \Big| |v_J(t)|^2 - |w_J(t)|^2 \Big|^{\frac{1}{2}} 
         + \Big| |w_J(t)|^2 - |w_J(0)|^2 \Big|^{\frac{1}{2}} \Big)
         \]
         \[
         + \sum_{|J| \leq M} \e^{sf(|J|)} \Big( \Big| |w_J(0)|^2 - |v_J(0)|^2 \Big|^{\frac{1}{2}} 
         + \Big| |v_J(0)|^2 - |u_J(0)|^2 \Big|^{\frac{1}{2}} \Big)
         \]
         \[
         + \sum_{|J| > M} \e^{sf(|J|)} \Big( \Big| |u_J(t)|^2 - |v_J(t)|^2 \Big|^{\frac{1}{2}} 
         + \Big| |v_J(t)|^2 - |v_J(0)|^2 \Big|^{\frac{1}{2}} 
         + \Big| |v_J(0)|^2 - |u_J(0)|^2 \Big|^{\frac{1}{2}} \Big).
         \]
         Finally, note the following useful property:
         \[
         \sum_{J \in \mathcal{Z}} \e^{sf(|J|)} \Big( |u_J|^2 - |u'_J|^2 \Big)^{\frac{1}{2}} 
         \leq \Big( \Vert u \Vert_s + \Vert u' \Vert_s \Big)^{\frac{1}{2}} \Vert u - u' \Vert_s^{\frac{1}{2}}.
         \]

    Applying this inequality to each term, along with the near-identity estimates from Propositions \ref{resonantlize} and \ref{integr}, we obtain:
    
    \begin{align*}
    	&\sum_{J\in\mathcal{Z}}\e^{sf(|J|)}| |u_J(t)|^2-|u_J(0)|^2|^{\frac{1}{2}}\\
    	\leq& (\Vert u(t)\Vert_s+\Vert v(t)\Vert_s)^{\frac{1}{2}}\Vert u(t)-v(t)\Vert_s^{\frac{1}{2}}+(\Vert w(t)\Vert_s+\Vert v(t)\Vert_s)^{\frac{1}{2}}\Vert w(t)-v(t)\Vert_s^{\frac{1}{2}}\\
    	&+(\Vert u(0)\Vert_s+\Vert v(0)\Vert_s)^{\frac{1}{2}}\Vert u(0)-v(0)\Vert_s^{\frac{1}{2}}+(\Vert w(0)\Vert_s+\Vert v(0)\Vert_s)^{\frac{1}{2}}\Vert w(0)-v(0)\Vert_s^{\frac{1}{2}}\\
    	&+\Vert u(0)\Vert_s^{\frac{3}{2}}+\Vert u(0)\Vert_s^{\frac{3}{2}}\\
    	\leq&(3\Vert u(t)\Vert_s)^{\frac{1}{2}}4\sqrt{C_K}\Vert u(t)\Vert^{\frac{3}{2}}+(3\Vert u(0)\Vert_s)^{\frac{1}{2}}4\sqrt{C_K}\Vert u(0)\Vert^{\frac{3}{2}}\\
    	&+(3\Vert v(t)\Vert_s)^{\frac{1}{2}}(4C_KC_2\Vert v(t)\Vert_s^3\gamma^{-2}d^3)^{\frac{1}{2}}+(3\Vert v(0)\Vert_s)^{\frac{1}{2}}(4C_KC_2\Vert v(0)\Vert_s^3\gamma^{-2}d^3)^{\frac{1}{2}}\\
    	&+\frac{1}{4}\Vert u(0)\Vert_s^{\frac{3}{2}}+\frac{1}{4}\Vert u(0)\Vert_s^{\frac{3}{2}}\\
    	<&\Vert u(0)\Vert_s^{\frac{3}{2}}.
    \end{align*}
     
     This result demonstrates that the bootstrap condition is strictly satisfied at time \(t = T\), which contradicts the definition of \(T\) as the supremum.  
     Therefore, our initial assumption \(T \leq T_r\) must be false, which implies \(T > T_r\).  
     This establishes the stability of the solution on the time interval \([0, T_r]\) and thereby completes the proof.

     \section{Measure Estimate}\label{sec:mea}
     In this section, we perform measure estimates to derive constraints on the parameters \(r\), \(\gamma\), and \(d\).  
     By combining the requirements for measure estimation with regularity balancing, we obtain the final estimate for the stability time \(T_r\).  
     
     For a fixed \(\mathcal{J} \in \mathfrak{J}^{2d,M}\), we define the resonant region:
     \[
     \mathfrak{R}_{\mathcal{J}} := \Big\{u \in B_{s}^M(1) \,\Big|\, |\omega_{\mathcal{J}}(u)| \leq 3\gamma \Big\}.
     \]
     The union of these regions
     \[
     \mathfrak{R}_{\gamma} = \bigcup_{\mathcal{J} \in \mathfrak{J}^{2d,M}} \mathfrak{R}_{\mathcal{J}}
     \]
     represents the resonant portion of the phase space, which must be excluded from our analysis.  
     The objective of this section is to demonstrate that the Lebesgue measure of \(\mathfrak{R}_{\gamma}\) is small for two physically meaningful choices of the kernel \(K\).

    \subsection{\(K_k = \frac{1}{|k|^p}\)}  
    \begin{lemma}  
    	Let
    	\[
    	K_k = \frac{1}{|k|^p}, \quad p \in \mathbb{Z}^+, \quad \text{for } |J| \neq 0, \quad \text{and set } K_0 = 0.
    	\]
    	For any \(\mathcal{J} \in \mathfrak{J}^{2d,M}\), there exists an index \(J^*\) with \(|J^*| \leq (p+1)d\) such that
    	\[
    	|\partial_{|u_{J^*}|^2} \omega_{\mathcal{J}}(u)| \geq \frac{1}{(4pd)^{2dp} \prod_{l=1}^{2d} |j_l|}.
    	\]
    \end{lemma}
    
    \begin{proof}
    	We compute
    	\[
    	|\partial_{|u_{J^*}|^2}\omega_{\mathcal{J}}| = \Big|\sum_{l=1}^{2d}\frac{\delta_l}{(j^* - j_l)^p}\Big| = \Big|P(j^*)\prod_{\alpha=1}^d \frac{1}{(j^* - j_l)^p}\Big|,
    	\]
    	where \(P(x)\) is a polynomial of degree strictly less than \(p(2d-1)\).  
    	
    	By the properties of polynomials, there exists an integer \(j^* \in \big(-(p+1)d, (p+1)d\big) \setminus \{j_1, j_2, \dots, j_{2d}\}\) such that \(P(j^*) \neq 0\), and specifically \(|P(j^*)| \geq 1\).  
    	
    	Now, consider the distances \(|j_l - j^*|\) for each \(j_l\):
    	If \(|j_l| \leq (p+1)d\), then \(|j_l - j^*| \leq 2(p+1)d\);
    	
    	 If \(|j_l| > (p+1)d\), then \(|j_l - j^*| \leq 2|j_l|\). In either case, we find that 
    	\[
    	|j_l - j^*| \leq 2|j_l|(p+1)d \leq 4pd|j_l|.
    	\]
    	Thus, we conclude
    	\[
    	|\partial_{|u_{J^*}|^2}\omega_{\mathcal{J}}| \geq \frac{1}{(4pd)^{2dp} \prod_{l=1}^{2d} |j_l|}.
    	\]
    \end{proof}
    
    Let \(J^*\) be the index guaranteed by the previous lemma. To apply Fubini's theorem, we consider the component \(|u_{J^*}|^2\) independently. Define \(u' \in B_s^M(1)\) such that \((u_J)_{J \neq J^*} = (u'_J)_{J \neq J^*}\), varying only the component \(u_{J^*}\).
    
    Since \(\omega_{\mathcal{J}}(u)\) depends on each action variable \(|u_J|^2\), we estimate:
    \[
    |\omega_{\mathcal{J}}(u) - \omega_{\mathcal{J}}(u')| \geq \frac{\big||u_{J^*}|^2 - |u'_{J^*}|^2\big|}{(4pd)^{2dp} \prod_{l=1}^{2d} |j_l|^p}.
    \]
    With the definition of \(\mathfrak{J}^{2d,M}\), this simplifies to:
    \[
    |\omega_{\mathcal{J}}(u) - \omega_{\mathcal{J}}(u')| \geq \frac{\big||u_{J^*}|^2 - |u'_{J^*}|^2\big|}{(4pdM)^{2dp}}.
    \]
    Thus, for \(u, u' \in \mathfrak{R}_{\mathcal{J}}\),
    \[
    \big||u_{J^*}|^2 - |u'_{J^*}|^2\big| \leq 3\gamma (4pdM)^{2dp}.
    \]
    
    Let \(B_{s,J^*}(1) := \{u \in B_s^M(1) \mid u_{J^*} = 0\}\) denote the projection of the ball \(B_s^M(1)\) onto the subspace with \(u_{J^*} = 0\). Using this projection, we estimate:
    \[
    \text{meas}(\mathfrak{R}_{\mathcal{J}}) \leq 6\pi \gamma (4pdM)^{2dp} \text{meas}(B_{s,J^*}(1)).
    \]
    
   We can calculate the measure of \(B_s^M(1)\) as follows:
   \begin{align*}
   	\text{meas}(B_s^M(1)) 
   	&= \int_{\substack{u' \in B_{s,J^*}(1), \, u_{J^*} \in \mathbb{C} \\ \Vert u'\Vert_s + \e^{sf(|J|)}|u_{J^*}| \leq 1}} 1 \, \d u' \, \d u_{J^*} \\
   	&= \e^{-sf(|J^*|)} \int_{0}^{1} \int_{u' \in (1-y)B_{s,J^*}(1)} 1 \, \d u' \, \d y \\
   	&= \text{meas}(B_{s,J^*}^M(1)) \e^{-sf(|J^*|)} \int_{0}^{1} \pi y (1-y)^{4d} \, \d y \\
   	&= \text{meas}(B_{s,J^*}^M(1)) \frac{\pi \e^{-sf(|J|)}}{4M(4M+1)}.
   \end{align*}
   
      Therefore, for \(|J^*| \leq (p+1)d\), we have the following upper bound:
      \[
      \text{meas}(\mathfrak{R}_{\mathcal{J}}) \leq \text{meas}(B_s^M(1)) \cdot 8M(4M+1) \cdot 3\gamma (4pdM)^{2dp} \e^{sf((p+1)d)}.
      \]
      
      Using the fact that \(\#\mathfrak{J}^{2d,M} \leq 2^{2d}(2M+1)^{2d} \leq (5M)^{2d}\), we can now estimate the total measure of \(\mathfrak{R}_\gamma\) as follows:
      \begin{align*}
      	\text{meas}(\mathfrak{R}_{\gamma}) &\leq \sum_{\mathcal{J} \in \mathfrak{J}^{2d,M}} \text{meas}(\mathfrak{R}_{\mathcal{J}}) \\
      	&\leq \text{meas}(B_s^M(1)) \cdot (5M)^{2d} \cdot 8M(4M+1) \cdot 3\gamma (4pdM)^{2dp} \e^{sf((p+1)d)} \\
      	&\leq \text{meas}(B_s^M(1)) \cdot (4pdM)^{2d(p+1)} \gamma \e^{sf((p+1)d)}.
      \end{align*}
      
      Thus, the inequality:
      \[
      \text{meas}(\mathfrak{R}_{\gamma}) \leq \kappa \cdot \text{meas}(B_s^M(1))
      \]
      holds under the condition that \(\gamma\) is chosen sufficiently small such that:
      \[
      (4pdM)^{2d(p+1)} \gamma \e^{sf((p+1)d)} \leq \kappa.
      \]

    \subsection{$K_k=\e^{-|k|^\beta}$}
    \begin{lemma}
    	When \(K_k = \e^{-|k|^\beta}\) with \(\beta \geq 1\), for \(\mathcal{J} \in \mathfrak{J}^{2d,M}\), there exists an index \(J^* \in \mathcal{J}\) such that:
    	\[
    	|\partial_{|u_{J^*}|^2} \omega_{\mathcal{J}}(u)| \geq C_e := \frac{\e - 2}{\e - 1}.
    	\]
    \end{lemma}
    \begin{proof}
    	Begin by computing:
    	\[
    	|\partial_{|u_{J^*}|^2} \omega_{\mathcal{J}}(u)| = \Big|\sum_{l=1}^{2d} \delta_l \e^{-|j^* - j_l|^\beta}\Big|.
    	\]
    	
    	Let \(J^*\) be one of the indices \(J_l\) such that \(\overline{J_l} \notin \mathcal{J}\). Then, we estimate:
    	\begin{align*}
    		\Big|\sum_{l=1}^{2d} \delta_l \e^{-|j^* - j_l|^\beta}\Big| 
    		&\geq 1 - \sum_{j_l \neq j^*} \e^{-|j^* - j_l|^{\beta}} \\
    		&\geq 1 - \sum_{l \geq 1} \e^{-l} \\
    		&\geq \frac{\e - 2}{\e - 1} = C_e.
    	\end{align*}
    \end{proof}

      Following the argument in the previous subsection, we can establish the following upper bounds:
      \[
      \text{meas}(\mathfrak{R}_{\mathcal{J}}) \leq \text{meas}(B_s^M(1)) \cdot 8M(4M+1) \cdot \gamma C_e \e^{sf((p+1)d)},
      \]
      and therefore:
      \[
      \text{meas}(\mathfrak{R}_\gamma) \leq \text{meas}(B_s^M(1)) \cdot 8M(M+1)(5M)^{2d} \cdot \gamma \e^{sf((p+1)d)} \leq \kappa \cdot \text{meas}(B_s^M(1)).
      \]
      
      Here, we have used the inequality:
      \[
      8M(M+1)(5M)^{2d} \cdot \gamma \e^{sf((p+1)d)} \leq \kappa.
      \]
      
      The final inequality holds by choosing \(\gamma\) sufficiently small, as discussed in the next section.

    \section{Time Length Analysis}\label{sec:len}
    We first consider the case where \(f(x) = x^{\mathfrak{g}}\) with \(0 < \mathfrak{g} < 1\), specifically corresponding to \(u(x)\) belonging to the Gevrey class.

    \subsection{$f(x)=x^{\mathfrak{g}},K_k=\frac{1}{|k|^p}$}
    
    The strictest constraint in this setting arises from the measure estimate, which requires:
    \[
    \gamma (4pdM)^{2d(p+1)} \e^{sf((p+1)d)} \leq \kappa.
    \]
    Given that \(\e^{s((p+1)d)^{\mathfrak{g}}} \ll (4pdM)^{2d(p+1)} = \e^{2d(p+1)\ln(4pdM)}\), the term \((4pdM)^{2d(p+1)}\) dominates these conditions.
    To counteract the effect of the dominating term, we choose \(\gamma\) as follows:
    \[
    \gamma = \kappa \e^{-2d(p+1)\ln(4pdM)}.
    \]
    We then select \(r\) to have a similar scale as \(\gamma\), introducing a parameter \(\iota > \frac{5}{2}\):
    \[
    r = \e^{-2\iota d(p+1)\ln(4pdM)}, \quad r = \Big(\frac{\gamma}{\kappa}\Big)^{\iota}.
    \]
    We then define \(\mathfrak{r}\) as:
    \[
    \mathfrak{r} = \frac{C_m r^{\frac{2}{5}}}{\gamma} = \frac{C_m}{\kappa} \e^{\Big(1 - \frac{2\iota}{5}\Big)2d(p+1)\ln(4pdM)}.
    \]
    
    To keep this term under control, we set:
    \[
    \kappa = r^a = \e^{-2a\iota d(p+1)\ln(4pdM)},
    \]
    where \(a < \Big(\frac{2}{5} - \frac{1}{\iota}\Big)\). This ensures that:
    \[
    \mathfrak{r} = C_m \e^{\Big(1 - \frac{2\iota}{5} + a\iota\Big)2d(p+1)\ln(4pdM)} < 1.
    \]
    Thus, the measure estimate for this situation is given by \(r^a\), where \(a < \frac{2}{5}\).
    
    Next, we determine the relationship among \(r\), \(d\), and \(M\) by balancing the contributions of the remainder terms. Specifically, we evaluate:
    \[
    \mathfrak{r}^d d^d = \e^{-\Big(\frac{M}{d}\Big)^{\frac{\mathfrak{g}}{2}}}.
    \]
    By substituting the expression for \(\mathfrak{r}\), we obtain the following equation that relates \(M\) and \(d\):
    \begin{align*}
    	C_m^d \e^{d^{\frac{\mathfrak{g}}{2}} 
    		\Big((1 - \frac{2\iota}{5} + a\iota) 2d^2(p+1)\ln(4pdM) + d\ln d\Big)} 
    	= \e^{-M^{\frac{\mathfrak{g}}{2}}}.
    \end{align*}
    
   Under the condition that:
   \[
   (1 - \frac{2\iota}{5} + a\iota)(p+1)d^2(\iota - 1)\ln 2 > d\ln d + d\ln\Big(\frac{C_m}{\kappa}\Big),
   \]
   it suffices to set:
   \[
   (1 - \frac{2\iota}{5} + a\iota)(\iota - 1) 2d^{2 + \frac{\mathfrak{g}}{2}}(p+1)\ln(2pdM) = M^{\frac{\mathfrak{g}}{2}}.
   \]
    By absorbing all constant factors into a single term, we arrive at the simplified expression:
    \[
    C_{a,p} d^{2 + \frac{\mathfrak{g}}{2}} \ln(dM) = M^{\frac{\mathfrak{g}}{2}},
    \]
    where the constant \(C_{a,p}\) is defined as:
    \[
    C_{a,p} = \Big(1 - \frac{2\iota}{5} + a\iota\Big)4(p+1).
    \]

    This equation can be explicitly solved for \(M\) using the Lambert W function. The solution is given as:
    \begin{align*}
    	\frac{2C_{a,p}}{\mathfrak{g}}d^{2+\mathfrak{g}}\ln(dM)^{\frac{\mathfrak{g}}{2}}&=(dM)^{\frac{\mathfrak{g}}{2}}=\exp(\ln(dM)^{\frac{\mathfrak{g}}{2}}),\\
    	-\ln(dM)^{\frac{\mathfrak{g}}{2}} \exp(-\ln(dM)^{\frac{\mathfrak{g}}{2}})&=-\frac{\mathfrak{g}}{2C_{a,p}d^{2+\mathfrak{g}}},\\ 	
    	-\ln(dM)^{\frac{\mathfrak{g}}{2}}&=W_{-1}(-\frac{\mathfrak{g}}{2C_{a,p}d^{2+\mathfrak{g}}}),\\
    	M&=\frac{1}{d}\exp(-\frac{2}{\mathfrak{g}}W_{-1}(-\frac{\mathfrak{g}}{2C_{a,p}d^{2+\mathfrak{g}}})).
    \end{align*}
    We now estimate the order of the remainder with respect to \(r\). Starting with \(|\ln r|\), we have:
    \[
    |\ln r| = 2\iota d(p+1)\Big(\ln(4p) - \frac{2}{\mathfrak{g}} W_{-1}\big(-\frac{\mathfrak{g}}{2C_{a,p}d^{2+\mathfrak{g}}}\big)\Big).
    \]
    For \(\ln|\ln r|\), it follows that:
    \[
    \ln|\ln r| = \ln d + \ln 2\iota(p+1) + \ln\Big(\ln(4p) - \frac{2}{\mathfrak{g}} W_{-1}\big(-\frac{\mathfrak{g}}{2C_{a,p}d^{2+\mathfrak{g}}}\big)\Big).
    \]
    We express \(\Big(\frac{M}{d}\Big)^{\frac{\mathfrak{g}}{2}}\) as:
    \begin{align*}
    	\Big(\frac{M}{d}\Big)^{\frac{\mathfrak{g}}{2}} 
    	&= \frac{1}{d^{\mathfrak{g}}} \exp\Big(-W_{-1}\big(-\frac{\mathfrak{g}}{2C_{a,p}d^{2+\mathfrak{g}}}\big)\Big)\\
    	&= -\frac{2}{\mathfrak{g}} C_{a,p} d^2 W_{-1}\big(-\frac{\mathfrak{g}}{2C_{a,p} d^{2+\mathfrak{g}}}\big).
    \end{align*}
    
    Finally, we compute the limiting behavior that establishes the relationship between \(T_r\) and \(r\):
    \begin{align*}
    &\lim\limits_{d\to\infty}\frac{(\frac{M}{d})^{\frac{\mathfrak{g}}{2}}\ln|\ln r|}{|\ln r|^2}\\
    =&\lim\limits_{d\to\infty}\frac{2C_{a,p}}{\mathfrak{g}(2\iota(p+1))^2}\frac{-d^2W_{-1}(-\frac{\mathfrak{g}}{2C_{a,p}d^{2+\mathfrak{g}}})(\ln2d\iota(p+1)+\ln(\ln(4p)-\frac{2}{\mathfrak{g}}W_{-1}(-\frac{\mathfrak{g}}{2C_{a,p}d^{2+\mathfrak{g}}})))}{(d(\ln(4p)-\frac{2}{\mathfrak{g}}W_{-1}(-\frac{\mathfrak{g}}{2C_{a,p}d^{2+\mathfrak{g}}})))^2}\\
  =&\lim\limits_{d\to\infty}\frac{C_{a,p}}{2\mathfrak{g}\iota^2(p+1)^2}\frac{(-\ln(\frac{\mathfrak{g}}{2C_{a,p}d^{2+\mathfrak{g}}}))\ln d}{d\ln(\frac{\mathfrak{g}}{2C_{\iota,p}d^{2+\mathfrak{g}}})^2}\\
    =&C_{gp}:=	\frac{C_{a,p}}{2\iota^2(p+1)^2\mathfrak{g}(2+\mathfrak{g})}.
    \end{align*}
    This asymptotic analysis justifies the choice of the stability time 
    \[
    T_r = \exp\bigg(C_{gp} \frac{|\ln r|^2}{\ln|\ln r|}\bigg)
    \]
    for this case.
    
    \subsection{$f(x) = x^{\mathfrak{g}}, K_k = \e^{-|k|^\beta}$}
    In this case, the constraint derived from the measure estimates is 
    \[
    8M(M+1)\gamma (5M)^{2d} \e^{s((p+1)d)^{\mathfrak{g}}} \leq \kappa.
    \]
    Based on this, we define:
    \[
    \gamma = \kappa \e^{-s((p+1)d)^{\mathfrak{g}}}M^{-3d} := \kappa \e^{-C_{s,p,\mathfrak{g}} d^{\mathfrak{g}}} M^{-3d}, 
    \quad r = \e^{-\iota s((p+1)d)^{\mathfrak{g}}}M^{-3\iota d}.
    \]

    Then, we define:
    \[
    \mathfrak{r} := \frac{C_m r^{\frac{2}{5}}}{\gamma} 
    = \frac{C_m}{\kappa} \e^{\big(1 - \frac{2\iota}{5}\big) C_{s,p,\mathfrak{g}} d^{\mathfrak{g}}} M^{\big(3 - \frac{6\iota}{5}\big)d}.
    \]
    
    Therefore, we can set:
    \[
    \kappa = r^a = \e^{-a\iota s((p+1)d)^{\mathfrak{g}}}M^{-3a\iota d},
    \]
    where \(a < \frac{2}{5} - \frac{1}{\iota}\).
    
    From this, we obtain that the measure estimate in this case is \(r^a\), for any \(a < \frac{2}{5}\).
    
    Next, we derive the relationship among \(r\), \(d\), and \(M\) by balancing the remainders:
    \[
    \mathfrak{r}^d d^d = \e^{-\Big(\frac{M}{d}\Big)^{\frac{\mathfrak{g}}{2}}}.
    \]
    Substituting \(\mathfrak{r}\) into the equation gives:
    \[
    C_m^d \e^{d^{\frac{\mathfrak{g}}{2}}\big((1 - \frac{2\iota}{5} + a\iota)C_{s,p,\mathfrak{g}}d^{\mathfrak{g}+1} + d\ln d\big)} 
    M^{\big(3 - \frac{6}{5}\iota + 3a\iota\big)d^{2+\frac{\mathfrak{g}}{2}}} = \e^{-M^{\frac{\mathfrak{g}}{2}}}.
    \]
    Since \(d^{1+\frac{3}{2}\mathfrak{g}} \ll d^{2+\mathfrak{g}} \ln M\), we let:
    \[
    \frac{1}{2} C_{s,p,\mathfrak{g},a} d^{2+\frac{\mathfrak{g}}{2}} \ln M = M^{\frac{\mathfrak{g}}{2}},
    \]
    where \(C_{s,p,\mathfrak{g},a} := 3 - \frac{6}{5}\iota + 3a\iota\).
    Thus, the expression for \(M\) in terms of \(d\) is obtained using the Lambert W function:
    \begin{align*}
    	\frac{1}{\mathfrak{g}}C_{s,p,\mathfrak{g},a}d^{2+\frac{\mathfrak{g}}{2}}\ln M^{\frac{\mathfrak{g}}{2}} 
    	&= M^{\frac{\mathfrak{g}}{2}} = \exp\big(\ln M^{\frac{\mathfrak{g}}{2}}\big), \\
    	-\ln M^{\frac{\mathfrak{g}}{2}} \exp\big(-\ln M^{\frac{\mathfrak{g}}{2}}\big)
    	&= -\frac{\mathfrak{g}}{C_{s,p,\mathfrak{g},a}d^{2+\frac{\mathfrak{g}}{2}}}, \\
    	-\ln M^{\frac{\mathfrak{g}}{2}}
    	&= W_{-1}\bigg(-\frac{\mathfrak{g}}{C_{s,p,\mathfrak{g},a}d^{2+\frac{\mathfrak{g}}{2}}}\bigg), \\
    	M
    	&= \exp\bigg(-\frac{2}{\mathfrak{g}} W_{-1}\bigg(-\frac{\mathfrak{g}}{C_{s,p,\mathfrak{g},a}d^{2+\frac{\mathfrak{g}}{2}}}\bigg)\bigg).
    \end{align*}
    Next, we carry out an analogous asymptotic analysis of the remainder term with respect to \(r\):

   \begin{align*}
   	(\frac{M}{d})^{\frac{\mathfrak{g}}{2}}&=-\frac{C_{s,p,\mathfrak{g},a}d^{2}}{\mathfrak{g}}W_{-1}\Big(-\frac{\mathfrak{g}}{C_{s,p,\mathfrak{g},a}d^{2+\frac{\mathfrak{g}}{2}}}\Big),\\
   	|\ln r|&=\iota s((p+1)d)^{\mathfrak{g}}-\frac{6\iota}{\mathfrak{g}} dW_{-1}(-\frac{\mathfrak{g}}{C_{s,p,\mathfrak{g},a}d^{2+\frac{\mathfrak{g}}{2}}}),\\
   	\lim\limits_{d\to\infty}\frac{(\frac{M}{d})^{\frac{\mathfrak{g}}{2}}(\ln|\ln r|)}{|\ln r|^{2}}&=C_{gb}:=\frac{C_{s,p,\mathfrak{g},a}}{\mathfrak{g}}(\frac{\mathfrak{g}}{6\iota})^{2}.
   \end{align*}
   
   Therefore, the stability time for this choice of kernel is expressed as:
   \[
   T_r = \exp\bigg(C_{gb} \frac{|\ln r|^2}{\ln |\ln r|}\bigg).
   \]

  Next, we analyze the case \(f(x) = (\ln x)^\theta\) with \(\theta > 1\), which corresponds to an ultra-differentiable function class. 
  From the choice of \(\kappa\), we also obtain a measure estimate of \(r^a\) for any \(a < \frac{2}{5}\).

    \subsection{$f(x)=(\ln x)^{\theta},K_k=\frac{1}{|k|^p}$}
    The parameter constraints are again derived from the measure estimates. Observing that \((\ln x)^\theta \ll x^{\mathfrak{g}}\), we adopt the same parameter choices for \(r\), \(\kappa\), and \(\gamma\) as in the previous subsection:
    \[
    r = \e^{-2\iota d(p+1)\ln(4pdM)}, \quad \gamma = \kappa \e^{-2d(p+1)\ln(4pdM)}, \quad \kappa = r^a, \quad \mathfrak{r} = C_m \e^{\big(\frac{1-\iota}{5} + a\iota\big) 2d(p+1)\ln(4pdM)},
    \]
    where \(a < \big(\frac{2}{5} - \frac{1}{\iota}\big)\).
    
    The relationship between \(r\), \(d\), and \(M\) can be expressed as:
    \[
    \mathfrak{r}^d d^d = \e^{\big(\ln\big(\frac{M}{d}\big)\big)^\theta}.
    \]
    When 
    \[
    \big(1 - \frac{2\iota}{5} + a\iota\big)d^2(p+1)\ln(4p) > d\ln\bigg(\frac{C_m}{\kappa}\bigg) + d\ln d,
    \]
    we can set
    \[
    \e^{\big(1 - \frac{2\iota}{5} + a\iota\big) 2d^2(p+1)\ln(dM)} = \e^{\big(\ln\frac{M}{d}\big)^\theta}.
    \]
    
    This leads to the simplified asymptotic relation:
    \[
    d^2 \big(\ln\big(\frac{M}{d}\big) + \ln d^2\big) = d^2 \ln(dM) = \big(\ln\frac{M}{d}\big)^\theta,
    \]
    which implies that \(\big(\ln\frac{M}{d}\big)^{\theta-1} > d^2\) and \(\ln\big(\frac{M}{d}\big) \gg \ln d^2.\) 
    
    Thus, by setting \(M = d\e^{d^{\frac{2}{\theta-1}}}\), we estimate the magnitude of the stability time:

    \begin{align*}
    	(\ln\frac{M}{d})^{\theta}&=d^{\frac{2\theta}{\theta-1}},\\
    	|\ln r|&=2\iota d(p+1)(\ln(4p)+\ln d^2+d^{\frac{2}{\theta-1}}),\\
    	\lim\limits_{d\to\infty}\frac{(\frac{M}{d})^{\theta}}{|\ln r|^{\frac{2\theta}{\theta+1}}}&=\lim\limits_{d\to\infty}\frac{d^{\frac{2\theta}{\theta-1}}}{(2\iota(p+1)(d^\frac{\theta+1}{\theta-1}+2\ln d+\ln4p))^{\frac{2\theta}{\theta+1}}}\\
    	&=C_{\theta p}:=\frac{1}{(2\iota(p+1))^{\frac{2\theta}{\theta+1}}}.
    \end{align*} 
    Thus, the stability time in this case is $T_r=\exp{(C_{\theta,p}|\ln r|^{
    \frac{2\theta}{\theta+1}})}$.

    \subsection{$f(x)=(\ln x)^{\theta},K_k=\e^{-|k|^\beta}$}
    In this case, the constraint from the measure estimates is given by:
    \[
    \gamma 8M(M+1)(5M)^{2d}\e^{(\ln(p+1)d)^{\theta}} \leq \kappa.
    \]
    Thus, we set:
    \[
    \gamma = \kappa M^{-3d}\e^{-(\ln(p+1)d)^{\theta}}, \quad 
    r = M^{-3\iota d}\e^{-\iota(\ln(p+1)d)^{\theta}}, \quad
    \kappa = r^a, \quad
    \mathfrak{r} = C_m M^{\big(1-\frac{2\iota}{5}+a\iota\big)3d} \e^{\big(1-\frac{2\iota}{5}\big)(\ln(p+1)d)^{\theta}},
    \]
    where $a<\frac{2}{5}-\frac{1}{\iota}$. Balancing the remainders gives:
    \[
    \mathfrak{r}^d d^d = \e^{-\big(\ln\frac{M}{d}\big)^{\theta}}.
    \] This implies:
    \[
    d\ln C_m + \big(1 - \frac{2}{5}\iota + a\iota\big)d^2\ln M + (1 - \iota)\big(\ln(p+1)d\big)^{\theta} = -\big(\ln\frac{M}{d}\big)^{\theta}.
    \]
    To simplify, we choose \(d\) such that:
    \[
    \big(-1 + \frac{2}{5}\iota - a\iota\big)d^2\ln M > \big(\ln(p+1)d\big)^{\theta} + \frac{1}{\iota-1}d\ln C_m.
    \]
    This gives the simplified asymptotic relation:
    \[
    d^2\big(\ln d + \ln\frac{M}{d}\big) = d^2\ln M = \big(\ln\frac{M}{d}\big)^{\theta}.
    \]
    This implies $(\ln M)^\theta\gg d^2,$ and consequently $\ln(\frac{M}{d})\gg\ln d$. We therefore set $M=d\e^{d^{\frac{2}{\theta-1}}}.$ Then 
    \begin{align*}
    	(\ln\frac{M}{d})^{\theta}&=d^{\frac{2\theta}{\theta-1}},\\
    	|\ln r|&=3\iota d(d^{\frac{2}{\theta-1}}+\ln d)+(\ln(p+1)d)^{\theta},\\
    	\lim\limits_{d\to\infty}\frac{(\ln\frac{M}{d})^{\theta}}{|\ln r|^{\frac{2\theta}{\theta+1}}}&=C_{\theta b}:=(\frac{1}{3\iota})^{\frac{2\theta}{\theta+1}}.
    \end{align*}

   Therefore, the stability time in this case is given by: $T_r=\exp({C_{\theta b}|\ln r|^{\frac{2\theta}{\theta+1}}})$.

	\appendix
	\renewcommand{\thesection}{\Alph{section}}
	\section*{Appendices}
	
	\section{Constants Index}
	\begin{align*}
		C_f&<1,\text{which appears in the setting of $f$},\\
		s_0&=\max\{\frac{\mathsf{d}}{2},\inf\{s\mid\sum_{J\in\mathcal{Z}}\e^{(2C_f-2)sf(|J|)}<\frac{1}{3}\},\\
		C_K&=\sup_{j\in\mathbb{Z}}|K_j|,\\
		C_2&=\frac{1}{48C_K},\\
		C_m&=4\max\{C_KC_2^2,1,C_K\},\\
		C_{a,p}&=(\frac{1-\iota}{5}+a\iota)4(p+1),\\	
		C_{gp}&=	\frac{C_{a,p}}{2\iota^2(p+1)^2\mathfrak{g}(2+\mathfrak{g})},\\
		C_{gb}&=\frac{3-3\iota+3s\iota}{\mathfrak{g}}(\frac{\mathfrak{g}}{6\iota})^{2},\\
		C_{\theta p}&=\frac{1}{(2\iota(p+1))^{\frac{2\theta}{\theta+1}}},\\
		C_{\theta b}&=(\frac{1}{3\iota})^{\frac{2\theta}{\theta+1}}.
	\end{align*}

	\section{Technical Lemmas }\label{sec:tec}
	\begin{lemma}[Lie bracket estimate]\label{Lie}
		Given two polynomials $P\in\mathcal{P}_{p},Q\in\mathcal{P}_{q},|Q|_{r,s}\leq\delta:=\frac{\rho}{8\e(r+\rho)},$ we have $\{P,Q\}\in\mathcal{P}_{p+q-2}$ and $|\{P,Q\}|_{r,s}\leq |P|_{r+\rho,s}|Q|_{r+\rho,s}\frac{1}{2\delta}$. Besides, 
		$$|ad_{Q}^kP|_{r,s}\leq|P|_{r+\rho,s}(\frac{|Q|_{r+\rho,s}}{2\delta})^k.$$
	\end{lemma}
	The proof can be seen in Appendix B in \cite{BMP20}.

	\begin{lemma}[Norm estimate for $P$]\label{normP}
		When $s>s_0$, for any $P\in\mathcal{P}_d,d\geq3,$ we have $$\Vert X_P\Vert_s\leq C_P\Vert u\Vert_s^{d-1},\ |P|_{r,s}\leq C_Pr^{d-2},$$ where $s_0$ satisfies $\sum_{J\in\mathcal{Z}}\e^{(2C_f-2)s_0f(|J|)}<\frac{1}{3}.$ 
	\end{lemma}
	\begin{proof}
		Let $P=\sum_{\mathcal{J}\in\mathcal{I}_d}P_{\mathcal{J}}u_{J_1}...u_{J_d},$ denote array $(J_1,...J_{k-1},(j,-1),J_{k+1},...,J_d)$ by $\hat{J}_{k,j}$, and denote $(J_1,...J_{k-1},J_{k+1},...,J_d)$ by $\hat{J}_k.$ Then
		\begin{align*}
			(X_P)_{j,+1}&=-\i\sum_{k=1}^{d}\sum_{\hat{J}_{k,j}\in \mathcal{I}_d}P_{\hat{J}_{k,j}}u_{\hat{J}_k},\\
			|(X_P)_{j,+1}|&\leq C_P\sum_{k=1}^{d}\sum_{\hat{J}_{k,j}\in \mathcal{I}_d}|u_{\hat{J}_k}|,\\
			|(X_P)_{j,+1}|e^{sf(\langle j\rangle)}&\leq C_P\sum_{k=1}^{d}\sum_{\hat{J}_{k,j}\in \mathcal{I}_d}|u_{\hat{J}_k}|e^{sf(|j|)}.
		\end{align*}
		Notice that $|j|=|\mathcal{M}(J_1,...,J_{k-1},J_{k+1},J_d)|,$ from $\mathcal{M}(\hat{J}_{k,j})=0.$   When $d\geq3$,  we have
		\begin{align*}
			sf(\langle j\rangle)\leq sf(\sum_{l\neq k}\langle J_l\rangle)\leq sf(\langle J_m\rangle)+sC_f(\sum_{l\neq m,k}\langle J_l\rangle).
		\end{align*} 
		We omit a technical discussion here. 
		Then 
		\begin{align*}
			|(X_P)_{j,+1}|e^{sf(\langle j\rangle)}&\leq C_P\sum_{k=1}^{d}\sum_{\hat{J}_{k,j}\in \mathcal{I}_d}
			e^{(1-C_f)sf(|J_m|)}\prod_{J\in\hat{J}_k}|u_{J}|e^{sC_f(\langle J\rangle)},\\
			\sum_{j\in\mathbb{Z}}|(X_P)_{j,+1}e^{sf(\langle j\rangle)}|^2&\leq C_P^2
			\sum_{j\in \mathbb{Z}} 
			\left(\sum_{k=1}^{d}\sum_{\hat{J}_{k,j}\in \mathcal{I}_d}
			e^{(1-C_f)sf(\langle J_m\rangle)}\prod_{J\in\hat{J}_k}|u_{J}|e^{sC_f(\langle J\rangle)}\right)^2\\
			&\leq d^2C_P^2(\sum_{J_m}|u_m|e^{sf(\langle J_m\rangle)})\prod_{J\neq J_m,J\in\hat{J}_k}\left(\sum_{J}|u_J|e^{sC_ff(\langle J\rangle )} \right)^2\\
			&\leq d^2C_P^2(\sum_{J_m}|u_m|e^{2sf(|J_m|)})\\
			&\prod_{J\neq J_m,J\in\hat{J}_k}(\sum_{J}|u_J|^2e^{2sf(|J|)})(\sum_{J}e^{(2C_f-2)s_0f(|J|)})^{d-2}\\
			&\leq \frac{d^2}{9^{d-2}}C_P^2\Vert u\Vert_s^{2d-2}.
		\end{align*}
		So $\Vert X_P\Vert_s\leq C_P\Vert u\Vert_s^{d-1}$ comes to the conclusion $|P|_{r,s}\leq C_Pr^{d-2}$.
	\end{proof}
	\begin{lemma}[Estimate for $W$ function]\label{Lam}
		For $x<-2,xe^x=y<0,x=W_{-1}(y),$ we have
		$$\lim\limits_{y\to0^-}\frac{\ln(-y)}{W_{-1}(y)}=1.$$
	\end{lemma}
	\begin{proof}
 Since $y=x\e^x,$ we have
	$$\lim\limits_{y\to0^-}\frac{\ln(-y)}{W_{-1}(y)}=\lim\limits_{x\to-\infty}\frac{\ln(-x)+x}{x}=1.$$
	\end{proof}

	\begin{lemma}\label{function}
		For $f(x)=x^{\theta}, f(x)\in\mathcal{F}$ with $c=2,C_f=2^{\theta-1}$.\\
		For $f(x)=(\ln x)^{q}, f(x)\in\mathcal{F}$ with $c=\max\{\e^{q-1},\exp(\frac{\ln2}{(\frac{3}{2})^{\frac{1}{q}}-1})\}\e^{q-1},C_f=\frac{1}{2}$. 
	\end{lemma}
	\begin{proof}
		For $f(x)=x^\theta$, we first prove the case of two variables:
		$$(x_1+x_2)^\theta\leq x_1^\theta+2^{\theta-1}x_2^\theta,x_1\geq2,x_2\geq2.$$ 
		By making the homogenizing substitution $t=\frac{x_2}{x_1}\leq1$, it suffices to prove:
		$$(1+t)^\theta\leq1+\frac{(2t)^\theta}{2},0<t\leq1,$$
		which is easy to prove.
		Then for $x_d<...<x_1$, we can get
		$$f(\sum_{l=1}^{d}x_l)\leq\sum_{l=1}^{d}C_f^{l-1}f(x_l)\leq f(x_1)+C_f\sum_{l=2}^{d}f(x_l).$$
		
		For $f(x)=(\ln x)^q,$ we take $F(x)=f(x+x_2)-f(x)$. Notice that 
		$$f''=q(\ln x)^{q-2}(\frac{q-1-\ln(x)}{x}),$$ so when $x>\e^{q-1}, f''<0$, $F'(x)=f'(x+x_2)-f'(x)<0$, we thus have $F(x_1)\leq F(x_2)=f(2x_2)-f(x_2)$. Hence we just need to prove:
		$$f(2x_2)-f(x_2)\leq \frac{1}{2} f(x_2),$$
		namely $$\frac{\ln (2x_2)}{\ln(x_2)}\leq(\frac{3}{2})^{\frac{1}{q}},$$
		which is hold for $x_2\geq\exp(\frac{\ln2}{(\frac{3}{2})^{\frac{1}{q}}-1})$ . Then we can also use induction to come to the conclusion.
	\end{proof}

	\section*{Declarations}
	
	\subsection*{Conflict of interest} On behalf of all authors, the corresponding author states that there is no conflict of interest.
	
	\subsection*{Data availability statements}
	My manuscript has no associated data.

	\section*{Acknowledgement}
	The second author (Y. Li) was supported by National Natural Science Foundation of China (12071175, 12471183 and 12531009).

\end{document}